\documentclass[11pt]{amsart}

 \usepackage[utf8]{inputenc}

 \usepackage[english]{babel}

 \usepackage{amsmath,amsfonts,amssymb}
 \usepackage{mathrsfs}
 \usepackage{dsfont} 
 \usepackage{mathtools} 
 \usepackage{esint} 
 \usepackage{xfrac} 
 \usepackage{defs} 
 \usepackage{fonts} 

 \usepackage{amsthm}
 \usepackage[shortlabels]{enumitem} 
 \usepackage[colorlinks=true,urlcolor=blue, citecolor=red,linkcolor=blue,linktocpage,pdfpagelabels, bookmarksnumbered,bookmarksopen]{hyperref} 
 \usepackage[margin=1in]{geometry} 

 \usepackage[linecolor=blue,backgroundcolor=blue!25,bordercolor=blue]{todonotes}
 \usepackage{cite} 

plus 6pt minus 12pt
\numberwithin{equation}{section}

\newtheorem{theorem}{Theorem}[section]
\newtheorem{proposition}[theorem]{Proposition}
\theoremstyle{definition}
\newtheorem{definition}[theorem]{Definition}
\theoremstyle{plain}
\newtheorem{lemma}[theorem]{Lemma}
\newtheorem{corollary}{Corollary}[theorem]
\theoremstyle{remark}
\newtheorem{remark}[theorem]{Remark}

\title[Level set estimate]{A note on estimates of level sets and their role in demonstrating regularity of solutions to nonlocal double phase equations}
\thanks{Support from NSF  DMS-1615726 is gratefully acknowledged.}
 \author{James M. Scott and Tadele Mengesha}
\address[James M. Scott]{Department of Mathematics,
University of Pittsburgh}
\address[Tadele Mengesha]{Department of Mathematics,
University of Tennessee Knoxville, mengesha@utk.edu}
\newcommand{\Kone}{\pi_1 \cK}
\newcommand{\Ktwo}{\pi_2 \cK}
\newcommand{\Kh}{\pi_h \cK}
\newcommand{\Hlambda}{ {\color{black} \cH_{\lambda}} }
\newcommand{\Hdlambda}{ {\color{black} \cH^d_{\lambda}} }
\newcommand{\Hndlambda}{ {\color{black} \cH^{nd}_{\lambda}} }
\newcommand{\GoodLambda}{ {\color{black} \mathscr{G}_{\lambda}} }
\newcommand{\BadLambdaD}{ {\color{black} \mathscr{B}_{\lambda, d}} }
\newcommand{\BadLambdaND}{ {\color{black} \mathscr{B}_{\lambda, nd}} }
\newcommand{\EquationReference}[2]{\mathrel{\overset{\makebox[0pt]{\mbox{\normalfont\tiny\sffamily #1}}}{#2}}}

\begin{document}

\maketitle

\begin{abstract}
	In this note we prove an estimate on the level sets of a function with $(p, q)$ growth that depends on the difference quotient of a bounded weak solution to a nonlocal double phase equation. This estimate is related to a self improving property of these solutions.
\end{abstract}

 \section{Introduction and Main Results}
 
This note is a companion to the article \cite{scott2020self}, in which the authors study regularity properties of weak solutions $u$ to 
 \begin{equation}\label{eq:Intro:MainEqn}
 	\cL u (x) = f(x)\,,
 \end{equation}
 where for measurable functions $u: \bbR^n \to \bbR$ and for $x \in \bbR^n$ the nonlocal \textit{double phase} operator $\cL$ is defined as
 \begin{equation*}
 	\cL u(x) := \pv \intdm{\bbR^n}{\frac{|u(x)-u(y)|^{p-2}}{|x-y|^{n+sp}} (u(x)-u(y)) + a(x,y) \frac{|u(x)-u(y)|^{q-2}}{|x-y|^{n+tq}} (u(x)-u(y))}{y}\,.
 \end{equation*}
 Throughout, we assume $n \geq 2$ and the integrability indices $p$, $q$ belong to $(1,\infty)$ with $p \leq q$ and differentiability indices $s$, $t$ belong to $(0,1)$. The abbreviation $\pv$ stands for principal value.
 The operator $\cL$ is the archetype of a class of nonlocal double phase operators first introduced in \cite{de2019holder}, in which the H\"older continuity of bounded viscosity solutions to $\cL u = f$ with bounded data $f$ was obtained. 
 In the work \cite{scott2020self} the authors demonstrate regularity of solutions on a different scale; that under suitable assumptions on the data $f$, the modulating coefficient $a(\cdot,\cdot)$, and a certain ratio of integrability and differentiability exponents solutions $u$ to $\cL u = f$ exhibit a self-improvement property. Precisely, distributional solutions $u$ belonging to the fractional Sobolev space $W^{s,p}(\bbR^n)$ in fact belong to a Sobolev space with higher exponents of integrability and differentiability.
 
 For ease of reference,  we summarize the relevant definitions, assumptions, properties and results found in \cite{scott2020self}.
 We assume that the modulating coefficient $a$ is measurable, and satisfies
 \begin{equation}\label{Assumption:Intro:Coeff}
 	a(x,y) \in L^{\infty}(\bbR^{2n})\,, \qquad 0 \leq a(x,y) \leq M\,, \qquad a(x,y) = a(y,x)\,.
 	\tag{A1}
 \end{equation}
 We also require that 
 \begin{equation}\label{Assumption:Intro:Exp}
 	p \leq q\,, \quad t \leq s\,, \qquad \frac{1}{p'} \leq \frac{tq}{sp} \leq 1\,, 
 	\tag{A2}
 \end{equation}
 where $p'$ is the H\"older conjugate of $p$: ${1\over p}+ {1\over p'}=1.$ 
 Additionally we will restrict ourselves to the case
 \begin{equation}\label{Assumption:Intro:Exp2}
 	s p < n\,.
 	\tag{A3}
 \end{equation}

  The higher differentiability and integrability described above and in \cite{scott2020self} applies to bounded solutions $u \in W^{s,p}(\bbR^n)$ of a weak formulation of the equation \eqref{eq:Intro:MainEqn}, that is
 \begin{equation}\label{eq:Intro:WeakFormulation}
 	\cE(u,\varphi) = \intdm{\bbR^n}{f(x) \varphi(x)}{x}\,, \qquad \text{ for any } \varphi \in C^{\infty}_c(\bbR^n)\,,
 \end{equation}
 where the form $\cE(u,\varphi)$ is defined as
 \begin{equation}\label{eq:Intro:WeakFormDefn}
 	\begin{split}
 		\cE(u,\varphi):= \iintdm{\bbR^n}{\bbR^n}{&\frac{|u(x)-u(y)|^{p-2}}{|x-y|^{n+sp}} (u(x)-u(y)) (\varphi(x)-\varphi(y)) \\
 			&\quad + a(x,y) \frac{|u(x)-u(y)|^{q-2}}{|x-y|^{n+tq}} (u(x)-u(y))(\varphi(x)-\varphi(y))}{y}{x}\,.
 	\end{split}
 \end{equation}
 %
 %
 %
 %
 %
 %
 We assume the data $f$ belongs to a Lebesgue space with sufficiently high exponent. Precisely, for a given $\delta_0 > 0$
 \begin{equation*}
 	f \in L^{p_{*_s}+\delta_0}_{loc}(\bbR^n)\,,
 \end{equation*}
 where we are using standard notation for H\"older and Sobolev exponents; that is, for any $r \in (1,\infty)$ and any $\sigma \in (0,1)$ we write
 \begin{equation*}
 	r' = \frac{r}{r-1}\,, \qquad r^{*} = r^{*_{\sigma}} = \frac{n r}{n-\sigma r}\,, \qquad r_* = r_{*_{\sigma}} = \frac{n r'}{n+\sigma r'} = (r^*)'\,.
 \end{equation*}
 (The dependence of the embedding exponents on $\sigma$ will be suppressed whenever it is clear from context.) 
 
 
 Weak solutions $u$ are assumed to be \textit{a priori} bounded, a point clarified by the following definition:
 \begin{definition}
 	A function $u \in W^{s,p}(\bbR^n) \cap L^{\infty}(\bbR^n)$ is a \textit{bounded weak solution} to \eqref{eq:Intro:MainEqn} with data $f$ if the nonlocal double phase energy $\cE(u,u) < \infty$ and if $u$ satisfies \eqref{eq:Intro:WeakFormulation}.
 \end{definition}
  If we denote the integrand of $\cE(u,u)$ by $P(x,y,u)$ so that 
  \begin{equation*}
 	\cE(u,u) = \iintdm{\bbR^n}{\bbR^n}{P(x,y,u)}{y}{x}\,,
 \end{equation*}
 then by definition of $u$ as a bounded weak solution the function $P(\cdot,\cdot,u)$ belongs to $L^1(\bbR^{2n})$.
 The following theorem concerning $P$ constitutes the main result of \cite{scott2020self}:
 
 \begin{theorem}\label{thm:Intro:MainThm}
 	Let $p$, $q \geq 2$ and $s$, $t \in (0,1)$ satisfy \eqref{Assumption:Intro:Exp}-\eqref{Assumption:Intro:Exp2} and let $a(x,y)$ satisfy \eqref{Assumption:Intro:Coeff}. Fix $\delta_0 > 0$, and let $f \in L^{p_{*_s} + \delta_0}_{loc}(\bbR^n)$. Let $u \in W^{s,p}(\bbR^n) \cap L^{\infty}(\bbR^n)$ be any bounded weak solution to \eqref{eq:Intro:MainEqn} with data $f$. Then there exists $\veps_0 \in (0,1)$ depending only on $n$, $p$, $q$, $s$, $t$, $M$, $\delta_0$ and $\Vnorm{u}_{L^{\infty}(\bbR^n)}$ such that for every $\tau \in (0,\veps_0)$
 	\begin{equation*}
 		P(\cdot,\cdot,u) \in L^{1+\tau}_{loc}(\bbR^{2n})\,.
 	\end{equation*}
 	In particular, there exist positive constants $\veps_1$ and $\veps_2$ such that $u \in W^{s+\veps_1, p+\veps_2}_{loc}(\bbR^n)$, and if $(s+\veps_1)(p+\veps_2) > n$ then $u$ is locally H\"older continuous.
 \end{theorem}

 To prove Theorem \ref{thm:Intro:MainThm} we use an argument developed by Kuusi, Mingione and Sire announced in \cite{kuusi2014fractional} and presented in \cite{kuusi2015} that builds a nonlocal fractional Gehring lemma in order to prove a self-improvement result for solutions to a class of monotone operators with quadratic growth related to the fractional Laplacian.
 The arguments in \cite{scott2020self} and in this note are heavily based on the work and presentation done for the case $p=2$ in \cite{kuusi2015}. While it is apparent from a careful reading of that work that their methods apply to functionals with more general $p$-growth, the precise treatment of such classes of operators does not appear in the literature. Since we are further working with operators of mixed $(p,q)$ growth, in \cite{scott2020self} and in this note we have written the arguments of \cite{kuusi2015} for a general exponent $p$ so that the robustness  of their technique and as well as results can be clearly seen as applicable in a wealth of contexts. One such instances is an extension of these arguments to vector-valued solutions of nonlocal systems. A specific example is the strongly coupled system of nonlinear equations studied in \cite{Mengesha-Scott-Korn-in-Domain}.  
 
  The fractional Gehring lemma relies on a level set estimate of a quantity related to the solution, and its proof is the main contribution of this note.  If a reader of \cite{scott2020self} accepts the level set estimate as true then the rest of the proof in \cite{scott2020self} of the Gehring lemma follows in a straightforward way. However, because the proof of the estimate itself is quite technical and very closely resembles the argument from \cite{kuusi2015}, we have written it here instead of in \cite{scott2020self}.
  
  Due to its technical nature we do not write a statement of the level set estimate precisely until the beginning of Section \ref{sec:LevelSetEstimate}. To ensure there is no interruption in the thread of reasoning between this work and \cite{scott2020self}, the prerequisite results are given in Section \ref{sec:RevHolder} as a summarized version of their counterparts in \cite{scott2020self}. 

 
 Following the structure introduced in \cite{kuusi2015}, we define \textit{dual pairs} of measures and functions $(U,\nu)$. For small $\veps \in (0,1/p)$  we define the locally finite doubling Borel measure in $\bbR^{2n}$
 \begin{equation}\label{eq:MeasureDefn}
 	\nu(A) := \int_{A} \frac{\rmd x \, \rmd y}{|x-y|^{n-\veps p}}\,, \qquad A \subset \bbR^{2n} \text{ measurable}\,,
 \end{equation}
 and we define the function
 \begin{equation}
 	U(x,y) := \frac{|u(x)-u(y)|}{|x-y|^{s+\veps}}\,.
 \end{equation}
 It is then clear that 
 \begin{equation*}
 	u \in W^{s,p}(\bbR^n) \qquad \text{ if and only if } \qquad U \in L^p(\bbR^{2n}; \nu)\,.
 \end{equation*}
 The integrand $P(x,y,u)$ of the energy $\cE(u,u)$ can be expressed in terms of $U$ as 
 \begin{equation}\label{eq:UDefn:A}
 	[U^p + A(x,y) U^q]|x-y|^{-n+\epsilon p},\quad \text{where $A(x,y) := a(x,y) |x-y|^{(s-t)q + \veps(q-p)}$. }
 \end{equation}
 %
 %
 We can therefore write the double phase energy $\cE(u.u)$ in terms the dual pair as 
 \begin{equation}\label{eq:UDefn:Energy}
 	\cE(u,u) = \intdm{\bbR^{2n}}{(U^p + A(x,y) U^q)}{\nu}=:\intdm{\bbR^{2n}}{G(x,y,U)}{\nu} 
 \end{equation}
 where the integrand $G(x,y,U):=U^p + A(x,y) U^q$. 
 Then it now becomes clear that 
 \begin{equation*}
 	P(\cdot,\cdot,u) \in L^1(\bbR^{2n}) \quad \text{ if and only if } \quad G(\cdot,\cdot,U) \in L^1(\bbR^{2n}; \nu)\,.
 \end{equation*}

 \begin{theorem}[Higher Regularity Result]\label{thm:Setup:MainRegularityResult}
 	With all the assumptions of Theorem \ref{thm:Intro:MainThm}, there exists $\veps_0 > 0$ depending only on $\texttt{data}$ such that for every $\delta \in (0,\veps_0)$ we have
 	\begin{equation}\label{eq:Setup:HigherRegForU}
 		G(x,y,U) \in L^{1 + \delta}_{loc}(\bbR^{2d}; \nu)\,.
 	\end{equation}
 	where $\texttt{data} $ represents  $n,p,q,s,t,M,$ and $\Vnorm{u}_{L^{\infty}}$.
 \end{theorem}

 Theorem \ref{thm:Intro:MainThm} is a simple consequence of the above theorem. 
 In \cite{scott2020self} we show \eqref{eq:Setup:HigherRegForU} directly
 by way of a fractional Gehring lemma applied to the dual pair of function and measure $(G,\nu)$.
 This fractional Gehring lemma in turn relies on a kind of fractional reverse H\"older inequality; for the exact statement see Theorem \ref{prop:RevHolder} below. This inequality holds only for diagonal sets of the type $B \times B \subset \bbR^{2n}$, and it is insufficient to apply tools traditionally used to prove Gehring's lemma such as the maximal function. Nevertheless, Kuusi, Mingione, and Sire in \cite{kuusi2015} used a novel localization technique to show that the fractional reverse H\"older-type inequality over diagonal balls is sufficient to prove a special fractional version of Gehring's lemma that is applicable for dual pairs of the above type. A key ingredient of this localization technique is a level set estimate 
 \begin{equation*}
 	\frac{1}{\lambda^2} \int_{\cB(x_0,\beta) \cap \{ U > \lambda \} } U^2 \, \rmd \nu \precsim \frac{1}{\lambda^r}  \int_{\cB(x_0,\alpha) \cap \{ U > \lambda \}} U^r \, \rmd \nu + \text{ terms involving level sets of } f,g
 \end{equation*}
 for some fixed $r<2$ and for any $\lambda \geq \lambda_0$, where $\lambda_0$ is a finite constant depending on the solution. Here $\cB = B \times B$, and $B \subset \bbR^{n}$ is a ball.
 We adapt the statement and proof of this level set estimate to our setting; see Proposition \ref{prop:LevelSetEstimate}. Key steps of the proof are described in Section \ref{sec:LevelSetEstimate}, and we additionally refer to the original discussions and summaries of the technique in \cite{kuusi2015, kuusi2014fractional}.
 
 We finally remark that the fractional Gehring lemma incorporates the level set estimate, and the proof of the fractional Gehring lemma itself is contained in \cite{scott2020self}. 
 
 This note is organized as follows: In the next section we identify notation and conventions. The reverse H\"older inequality is proved in \cite{scott2020self}, and for reference is stated in Section \ref{sec:RevHolder}. Section \ref{sec:LevelSetEstimate} contains the proof of the level set estimate in its entirety.

 \section{Preliminaries}
 

 Throughout, we denote positive constants by $c$, $C$, etc., and they may change from line to line. We list the dependencies in parentheses after the constant when we wish to make them explicit, i.e.\ if a constant $C$ depends only on $n$, $p$ and $s$, we write $C = C(n,p,s)$. We will abbreviate the following set of parameters as
 \begin{equation*}
 	\texttt{data} \equiv (n,p,q,s,t,M,\Vnorm{u}_{L^{\infty}})\,.
 \end{equation*}
 In $\bbR^n$, denote the open ball of radius $R$ centered at $x_0$ by
 \begin{equation*}
 	B(x_0,R) = B_R(x_0) := \{ x \in \bbR^n \, : \, |x-x_0| < R \}\,.
 \end{equation*}
 We will sometimes denote the ball $B \equiv B_R \equiv B_R(x_0)$ whenever the center and/or radius is clear from context. If $B$ is a ball centered at $x_0$ with radius $R$, then $\sigma B$ is the ball centered at $x_0$ with radius $\sigma R$.
 Given any measure $\mu$, denote the average of a $\mu$-measurable function $h$ over a set $\cA$ by
 \begin{equation*}
 	(h)_{\cA} := \fint_{\cA} h \, \rmd \mu = \frac{1}{\mu(\cA)} \intdm{\cA}{h(x)}{\mu}\,.
 \end{equation*}
 In dealing with functions defined on $\bbR^{2n}$ such as $U$, we consider the norm on $\bbR^{2n}$ defined by
 \begin{equation*}
 	\Vnorm{(x,y)} := \max \{ |x|, |y| \}\,,
 \end{equation*}
 where $|\cdot|$ denotes the Euclidean norm on $\bbR^n$.
 Denote the balls defined by this norm as
 \begin{equation*}
 	\begin{split}
 		\cB(x_0,y_0,R) &:= \{ (x,y) \in \bbR^n \times \bbR^n \, : \, \Vnorm{(x,y)-(x_0,y_0)} < R\} \\
 		&= B(x_0,R) \times B(y_0,R)\,.
 	\end{split}
 \end{equation*}
 If we denote
 \begin{equation*}
 	B_{\bbR^{2n}}(x_0,y_0,R) := \{ (x,y) \in \bbR^n \times \bbR^n \, : \, \sqrt{|x-x_0|^2 + |y-y_0|^2} < R \}\,,
 \end{equation*}
 then clearly
 \begin{equation*}
 	B_{\bbR^{2n}}(x_0,y_0,R) \subset \cB(x_0,y_0,R) \subset B_{\bbR^{2n}}(x_0,y_0,2R)\,. 
 \end{equation*}
 Often we will need to consider balls in $\bbR^{2n}$ centered at a point on the ``diagonal," that is, a point of the form $(x_0,x_0)$ for $x_0 \in \bbR^n$. In this case we abbreviate $\cB(x_0,x_0,R) \equiv \cB(x_0,R)$. We will also use the abbreviations $\cB(x_0,R) \equiv \cB_R(x_0) \equiv \cB_R \equiv \cB$ whenever the center and/or radius is clear from context. Whenever there is no ambiguity we write $\cB(x_0,\sigma R) = \sigma \cB$.
 We also denote
 \begin{equation*}
 	\text{Diag} := \{ (x,x) \, : \, x \in \bbR^n \}\,.
 \end{equation*}
 We will use the elementary inequality
 \begin{equation}\label{eq:GeometricSeriesEstimate}
 	2^{k r} \sum_{j=k-1}^{\infty} 2^{-j r} \leq \frac{4^{ r}}{r \ln(2)}\,, \qquad \text{ for } k \geq 1 \quad \text{and} \quad r \in (0,\infty)\,.
 \end{equation}
 The cardinality of a finite set $\cA$ is denoted by $\# \cA$. The set of nonnegative integers $\{ 0,1,2, \ldots \}$ is designated by $\bbZ_+$.

 For any domain $\Omega \subset \bbR^n$, $0 < \sigma <1$ and $r \in [1,\infty)$ the fractional Sobolev spaces are defined by the Gagliardo seminorm
 \begin{equation*}
 	W^{\sigma,r}(\Omega) := \left\{ u \in L^r(\Omega) \, : \, [u]_{W^{\sigma,r}(\Omega)} :=  \iintdm{\Omega}{\Omega}{\frac{|u(x)-u(y)|^r}{|x-y|^{n+\sigma r}} }{y}{x} < \infty \right\}
 \end{equation*}
 with norm $\Vnorm{\cdot}_{W^{\sigma,r}(\Omega)}^r :=  \Vnorm{\cdot}_{L^{r}(\Omega)}^r + [\cdot]_{W^{\sigma,r}(\Omega)}^r$.
 
 We will also use the following fractional Poincar\'e-Sobolev-type inequalities throughout the note. A proof of the first can be found in several places; see for instance \cite{DNPV12, bass2005holder}. The second can be found in \cite{mingione2003singular, schikorra2016nonlinear}.
 
 \begin{theorem}[Fractional Poincar\'e-Sobolev Inequality]\label{thm:SobolevInequality}
 	Let $r \in [1,\infty)$, $0<\sigma<1$. Let $B = B_R(x_0)$ for some $R>0$, $x_0 \in \bbR^n$. Then there exists $C = C(n,r,\sigma) > 0$ such that 
 	\begin{equation*}
 		\left( \fint_B \left| \frac{v(x) - (v)_B}{R^{\sigma}} \right|^{r^{*_{\sigma}}} \, \rmd x \right)^{1/{r^{*_{\sigma}}}} \leq C \left( \int_B \fint_B \frac{|v(x)-v(y)|^r}{|x-y|^{n+\sigma r}} \, \rmd y \, \rmd x \right)^{1/r}
 	\end{equation*}
 	for every $v \in W^{\sigma,r}(B)$.
 \end{theorem}
 
 \begin{theorem}[Fractional Poincar\'e Inequality]\label{thm:PoincareInequality}
 	Let $r \in [1,\infty)$, $0<\sigma<1$. Let $B = B_R(x_0)$ for some $R>0$, $x_0 \in \bbR^n$. Then there exists $C = C(n,r) > 0$ such that 
 	\begin{equation*}
 		\left( \fint_B \left| \frac{v(x) - (v)_B}{R^{\sigma}} \right|^{r} \, \rmd x \right)^{1/r} \leq C \left( \int_B \fint_B \frac{|v(x)-v(y)|^r}{|x-y|^{n+\sigma r}} \, \rmd y \, \rmd x \right)^{1/r}
 	\end{equation*}
 	for every $v \in W^{\sigma,r}(B)$.
 \end{theorem}
 
\section{Sobolev Inequality for Dual Pairs and Reverse H\"older Inequality}\label{sec:RevHolder}

\subsection{The Dual Pair Measure}

We summarize some basic properties of the measure $\nu$ defined in \eqref{eq:MeasureDefn}. These properties are natural extensions of those established in \cite[Proposition 4.1]{kuusi2015}; their proof is sketched in \cite{scott2020self}. 
\begin{theorem}\label{thm:Measure}
For any $\veps \in (0,1/p)$, the measure $\nu$ defined as
\begin{equation*}
\nu(\cA) := \intdm{\cA}{\frac{1}{|x-y|^{n-\veps p}}}{y}\, \rmd x\,, \qquad \cA \subset \bbR^{2n}\,,
\end{equation*}
is absolutely continuous with respect to Lebesgue measure on $\bbR^{2n}$. Additionally,
\begin{itemize}
\item For $\cB = B_R(x_0) \times B_R(x_0)$,
\begin{equation}\label{eq:MeasureOfBall}
\nu(\cB) = \frac{c(n,p,\veps) R^{n+\veps p}}{\veps}\,,
\end{equation}
where $c(n,p,\veps)$ is a constant depending only on $n$, $p$ and $\veps$ that satisfies $1/\widetilde{c}(n,p) \leq c(n,p,\veps) \leq \widetilde{c}(n,p)$, where $\widetilde{c}$ is another constant depending only on $n$ and $p$. 

\item For every $x \in \bbR^n$ and for $R \geq r > 0$, \begin{equation}\label{eq:MeasureDoublingProperty}
\frac{\nu( \cB(x,R))}{\nu( \cB(x,r))} = \left( \frac{R}{r} \right)^{n+\veps p}\,.
\end{equation}

\item For every $a \leq 1$, $R> 0$ and $x \in \bbR^n$, there exists a constant $C_d = C_d(n,p)$ such that
\begin{equation}\label{eq:MeasureOfBallAndCube}
\frac{\nu(\cB(x,R))}{\nu(K_1 \times K_2)} \leq \frac{C_d}{a^{2n} \veps}
\end{equation}
for any two cubes $K_1$, $K_2 \subset B_R(x)$ with sides parallel to the coordinate axes and such that $|K_1| = |K_2| = (aR)^n$.
\end{itemize}
\end{theorem}

\subsection{Reverse H\"older Inequality} Recall that 
\begin{equation}\label{eq:DefnOfUandF}
U(x,y) = \frac{|u(x)-u(y)|}{|x-y|^{s+\veps}}\,, \qquad\text{and define\,\,} F(x,y) := |f(x)|\,.
\end{equation}
Then $F \in L^{p_* + \delta}_{loc}(\bbR^{2n})$ for every $\delta \in (0,\delta_0)$, as a direct calculation using the properties of measure $\nu$.

We now report the compatibility of the Sobolev-Poincar\'e inequality with the definition of $U$. Given $B = B_R(x_0)$, define $\tau \in (0,1)$, and $\eta \in (1,\infty)$ to be differentiability and integrability constants respectively that have yet to be fixed. Letting $ \textstyle \veps \in (0,\min\{{s\over p}, 1-s\})$ and using \eqref{eq:MeasureOfBall},
\begin{equation*}
\fint_B \int_B \frac{\left| u(x) - u(y) \right|^{\eta}}{|x-y|^{n+\tau {\eta}}} \, \rmd y \, \rmd x =  \frac{C R^{\veps p}}{\veps}\fint_{\cB} U^{\eta} \, \rmd \nu
\end{equation*}
so long as
\begin{equation*}
\tau + \frac{\veps p}{\eta} = s + \veps\,.
\end{equation*}
Since $ \textstyle \veps \in (0,{s\over p})$ and $\veps < 1-s$ the exponent $\tau$ remains in $(0,1)$ for every $\eta \in (1,\infty)$. With this choice of $\tau$, by the fractional Sobolev inequality
\begin{equation*}
\left( \fint_B \left| \frac{u(x) - (u)_B}{R^{\tau}} \right|^m \, \rmd x \right)^{1/m} \leq C \left( \fint_B \int_B \frac{\left| u(x) - u(y) \right|^{\eta}}{|x-y|^{n+\tau \eta}} \, \rmd y \, \rmd x \right)^{1/\eta}
\end{equation*}
for every $m \in [1,\eta^{*_{\tau}}]$ with $\eta \in (1,\infty)$.
We choose $\eta$ to satisfy the relation
\begin{equation}\label{eq:DefnOfEta}
p = \eta^{*_{\tau}} = \frac{n \eta}{n - \tau \eta} = \frac{n \eta}{n- \eta(s + \veps - \frac{\veps p}{\eta})} \quad \Longleftrightarrow \quad \eta = \frac{np  +  \veps p^2}{n + sp + \veps p}\,.
\end{equation}
This choice of $\eta$ is a valid Lebesgue exponent; note that $\eta < p$ for all $n \geq 2$ and for all $p \in (1,\infty)$, and that $\eta > 1$ so long as $p \geq 2$.
Taking $m=\eta^{*_{\tau}}$ we summarize this discussion in the following lemma:
\begin{lemma}\label{lma:SobolevEmbeddingForDualPair}
Let $\veps \in (0,s/p)$ with $\veps < 1-s$ and $p \geq 2$.
Define $\eta = \frac{np  +  \veps p^2}{n + sp + \veps p}$. Then 
\begin{equation*}
\left( \fint_B \left| u(x) - (u)_B \right|^p \, \rmd x \right)^{1/p} \leq \frac{C R^{s+\veps}}{\veps^{1/\eta}} \left( \fint_{\cB} U^{\eta} \, \rmd \nu \right)^{1/\eta}\,,
\end{equation*}
where $C = C(n,s,p)$. The same inequality holds when the ball $B$ is replaced by a cube $Q$ with sides of length $R$ and with $\cB$ replaced by $Q \times Q$.
\end{lemma}


Recall that $G(x, y, U) =U^p + A(x,y) U^q$. We have the following $L^1_{loc}$ estimate for $G$ which will lead us to a scale-invariant reverse H\"older's inequality.  The statement is precisely \cite[Proposition 4.3]{scott2020self} and its proof can be found in the same paper. 
\begin{proposition}\label{prop:RevHolder}
Let $p \in [2,\infty)$, and let $\veps < 1-s$ with $ \textstyle \veps \in ( 0, \min \{ s( \frac{tq}{sp} - \frac{1}{p'}) , \frac{s}{p} \} )$. (This choice is possible by Assumption \ref{Assumption:Intro:Exp}). Let $\eta$ be given by the formula  in \eqref{eq:DefnOfEta}, Let $B = B_R(x_0)$ be a ball with $R \leq 1$. Then there exists a constant $C$ depending only on $\texttt{data}$ such that for any solution $u \in W^{s,p}(\bbR^n) \cap L^{\infty}(\bbR^n)$ to \eqref{eq:Intro:MainEqn} and for any $\sigma \in (0,1)$
\begin{equation}\label{eq:RevHolder}
\begin{split}
\left( \fint_{\frac{1}{4}\cB} G(x,y,U) \, \rmd \nu \right)^{1/p} &\leq \frac{C}{\veps^{1/\eta - 1/p}}\Bigg[\frac{1}{\sigma } \left( \fint_{\cB} U^{\eta} \, \rmd \nu \right)^{1/\eta} \\
	&\quad + { \sigma}\sum_{k=0}^{\infty} 
	\big( 2^{-k (\frac{sp}{p-1} -s - \veps)} +  2^{-k (\frac{tq}{p-1} -s - \veps)} \big) 
	\left( \fint_{2^k \cB} U^{\eta} \, \rmd \nu \right)^{1/\eta}\Bigg]\\
	&+ \frac{C [\veps \nu(\cB)]^{\frac{\theta}{p-1}}}{\veps^{  (1/p_* - 1/p') \frac{1}{p-1}  }} \left[ \left( \fint_{\cB} F^{p_*} \, \rmd \nu \right)^{1/p_*} \right]^{1/(p-1)}\,, \\
\end{split}
\end{equation}
where
\begin{equation*}
\theta := \frac{s-\veps(p-1)}{n+\veps p} > 0\,.
\end{equation*}
\end{proposition}

\begin{remark}
We make some remarks. The upper bound in \eqref{eq:RevHolder} can be simplified down to just one series. Since $sp \geq tq$
\begin{equation}\label{eq:SeriesComparisons}
	\begin{split}
		2^{-k(\frac{sp}{p-1} -s -\veps)} \leq 2^{-k(\frac{tq}{p-1} - s - \veps)}\,, \qquad k \in \bbZ_+\,,
	\end{split}
\end{equation} 
so we can replace the infinite series on the right-hand side of \eqref{eq:RevHolder} with
\begin{equation*}
	\frac{C \sigma}{\veps^{1/\eta-1/p}} \sum_{k = 0}^{\infty} \alpha_k \Vparen{\fint_{2^k \cB} U^{\eta} \, \rmd \nu }^{1/\eta}\,,
\end{equation*}
where
\begin{equation}\label{eq:DefnOfGeometricSeries}
	\alpha_k := 2^{-k(\frac{tq}{p-1} - s - \veps)}\,.
\end{equation}
Moreover, in the case $a \equiv 0$ one simply takes $\alpha_k = 2^{-k(\frac{s}{p-1}-\veps)}$.
In any case, since $ \textstyle \veps \leq \min \{ s( \frac{tq}{sp} - \frac{1}{p'}) , \frac{s}{p} \} $ the series $\sum_{k=0}^{\infty}\alpha_k <\infty$ and as a consequence  
\[
\begin{split}
	\sum_{k = 0}^{\infty} \alpha_k \Vparen{\fint_{2^k \cB} U^{\eta} \, \rmd \nu }^{1/\eta} &\leq \sum_{k = 0}^{\infty} \alpha_k \Vparen{\fint_{2^k \cB} U^{p} \, \rmd \nu }^{1/p}\\
	&=C(\veps,p,s)\sum_{k = 0}^{\infty} \alpha_k  \Vparen{\int_{2^k B}\int_{2^{k}B} {|u(y)-u(x)|^{p}\over |x-y|^{n+sp}} \, \rmd x \rmd y }^{1/p}\\
	& \leq  R^{-n/p-\epsilon}C(\veps,p,s)\Vparen{\int_{\mathbb{R}^{n}}\int_{\mathbb{R}^{n}} {|u(y)-u(x)|^{p}\over |x-y|^{n+sp}} \, \rmd x \rmd y }^{1/p} < \infty. 
\end{split}
\]

\end{remark}

The following corollary establishes a genuine scale-invariant reverse H\"older inequality for an appropriately scaled version of the integrand $G$. This quantity will satisfy a self-improving result. 
\begin{corollary}\label{cor:RevHolderFinal}
	Let $\textstyle \veps \in \left( 0, \min \{ s( \frac{tq}{sp} - \frac{1}{p'}) , \frac{s}{p} \} \right)$. (This choice is possible by Assumption \ref{Assumption:Intro:Exp}). Let $B = B_R(x_0)$ be a ball with $R \leq 1$. Define $H(x,y,U) := G(x,y,U)^{(p-1)/p}$. Then there exists a constant $C$ depending only on $\texttt{data}$ such that for any solution $u \in W^{s,p}(\bbR^n) \cap L^{\infty}(\bbR^n)$ to \eqref{eq:Intro:MainEqn} and for any $\sigma \in (0,1)$
	\begin{equation}\label{eq:RevHolderFinal}
		\begin{split}
			\left( \fint_{\frac{1}{4} \cB} H(x,y,U)^{p'} \, \rmd \nu \right)^{1/p'} &\leq \frac{C}{\sigma \veps^{1/\gamma - 1/p'}} \left( \fint_{\cB} H(x,y,U)^{\gamma} \, \rmd \nu \right)^{1/\gamma} \\
			&\qquad + \frac{C \sigma}{ \veps^{1/\gamma - 1/p'}} \sum_{k=0}^{\infty} \alpha_k \left( \fint_{2^k \cB} H(x,y,U)^{\gamma} \, \rmd \nu \right)^{1/\gamma} \\
			&+ \frac{C [\nu(\cB)]^{\theta}}{\veps^{1/p_* - 1/p'}} \left( \fint_{\cB} F^{p_*} \, \rmd \nu \right)^{1/p_*}\,, \\
		\end{split}
	\end{equation}
	where $\gamma := \frac{\eta}{p-1} =  p' \cdot  \frac{n+\veps p}{n+sp+\veps p} < p'$ and
	$
	\theta := \frac{s-\veps(p-1)}{n+\veps p}
	$.
\end{corollary}

\begin{remark}
	If $a \equiv 0$ one can see from careful inspection of the proofs they need not assume $u \in L^{\infty}(\bbR^n)$ in Proposition \ref{prop:RevHolder} and Corollary \ref{cor:RevHolderFinal}.
\end{remark}

\section{Proof of the Level Set Estimate}\label{sec:LevelSetEstimate}

We are now ready to state and prove the level set estimate. It is stated precisely in Proposition \ref{prop:LevelSetEstimate} below, and this section is devoted to its proof. We first define the following: using the notation of Corollary \ref{cor:RevHolderFinal}, for any $x_0 \in \bbR^n$ and $R > 0$ set
\begin{equation}\label{eq:DefnOfTheta}
	\Theta(x_0,R) := \Upsilon_0(x_0, R) + Tail(x_0, R) + \Psi_1(x_0,R)\,,
\end{equation}
where
\begin{equation}\label{eq:DefinitionOfTails1}
	\begin{split}
		\Upsilon_0(x_0,R) &:= \left( \fint_{\cB(x_0,R)} F^{p_*+\delta_f} \, \rmd \nu \right)^{1/(p_* + \delta_f)}\,, \text{ with } \delta_f \in (0,\delta_0) \text{ to be determined,} \\
		Tail(x_0,R) &:= \sum_{k=0}^{\infty} 2^{-k(\frac{tq}{p-1}-s-\veps)} \left( \fint_{\cB(x_0,2^k R)} H^{\gamma} \, \rmd \nu \right)^{1/\gamma}\,,
	\end{split}
\end{equation}
and, for any constant $M \geq 1$,
\begin{equation}\label{eq:DefinitionOfTails2}
	\Psi_M(x_0,R) := \left( \fint_{\cB(x_0,R)} H^{p'} \, \rmd \nu \right)^{1/p'} + M \frac{[\nu(\cB(x_0,R))]^{\theta}}{\veps^{1/p_* - 1/p'}} \left( \fint_{\cB(x_0,R)} F^{p_*} \, \rmd \nu \right)^{1/p_*}\,;
\end{equation}
we write $\Psi_M$ with $M=1$ as $\Psi_1$.

\begin{proposition}\label{prop:LevelSetEstimate}
Assume \eqref{Assumption:Intro:Coeff}, \eqref{Assumption:Intro:Exp}, and \eqref{Assumption:Intro:Exp2}. Assume that $\veps > 0$ satisfies
\begin{equation}\label{Assumption:Epsilon}
	\veps \in (0,s/p)\,, \qquad \veps < s \left( \frac{tq}{sp} -\frac{1}{p'} \right)\,, \qquad \veps < 1-s\,. \tag{A4}
\end{equation} 
Let $u \in W^{s,p}(\bbR^n) \cap L^{\infty}(\bbR^n)$ be a bounded weak solution to \eqref{eq:Intro:MainEqn}, and let $f \in L^{p_* + \delta_0}(\bbR^n)$ for given $\delta_0 > 0$. Let $U$ and $F$ be as in \eqref{eq:DefnOfUandF}.
Let $\cB(x_0,\varrho_0) \subset \bbR^{2n}$ with $0 < \varrho_0 \leq 1$, and let $\alpha$ and $\beta$ be such that $\varrho_0 < \beta < \alpha < \frac{3}{2} \varrho_0$ so that we have
\begin{equation*}
\cB(x_0,\varrho_0) \subset \cB(x_0,\beta) \subset \cB(x_0,\alpha) \subset \cB \big(x_0,\frac{3}{2} \varrho_0 \big)\,.
\end{equation*}
Then there exist constants $C_{\alpha} = C_{\alpha}(\texttt{data}) >0$, $C_f = C_f(\texttt{data},\veps) \geq 1$ and $\kappa_f = \kappa_f(\texttt{data},\veps) \in (0,1)$, with positive constants
\begin{equation}\label{eq:DefnOfLevelSetConstants}
\vartheta :=
\frac{3(p'-\gamma)}{\gamma}\,,
	\qquad
\vartheta_f := 
(p_* + \delta_f) \left( \frac{p_* \theta}{1-p_* \theta} \right)\,,
	\qquad 
\widetilde{\vartheta}_f := 
\frac{p_*(1+\theta \delta_f)}{1-p_* \theta}\,,
\end{equation}
such that
\begin{equation}\label{eq:LevelSetEstimate}
\frac{1}{\lambda^{p'}} \intdm{\cB(x_0,\beta) \cap \{ H > \lambda \} }{H^{p'}}{\nu} \leq \frac{C_{\alpha}}{\veps^{\vartheta} \lambda^{\gamma}} \intdm{\cB(x_0,\alpha) \cap \{ H > \lambda \} }{H^{\gamma}}{\nu} + \frac{C_f \lambda_0^{\vartheta_f}}{\lambda^{\widetilde{\vartheta}_f}} \int_{\cB(x_0,\alpha) \cap \{ F > \kappa_f \lambda \} } F^{p_*} \, \rmd \nu
\end{equation}
for every $\lambda \geq \lambda_0$, where $\lambda_0$ is defined as
\begin{equation}\label{eq:DefnOfLambda0}
\lambda_0 := \frac{C_a}{\veps} \left( \frac{\varrho_0}{\alpha-\beta} \right)^{2n+p} \Theta(x_0,2 \varrho_0)\,,
\end{equation}
and where $C_a = C_a(\texttt{data})$ and $\Theta$ has been defined in \eqref{eq:DefnOfTheta}. (see also \eqref{eq:DefinitionOfTails1} and \eqref{eq:DefinitionOfTails2}).

\end{proposition}
For $p=2$, this proposition is proved in Section 5 of \cite{kuusi2015}. Our proof is essentially the same as the proof found in \cite{kuusi2015}. 
However, we are writing the proof to make sure that the choice of the other parameters in \eqref{eq:DefnOfLevelSetConstants} are correctly made and to emphasize the robustness of the arguments in \cite{kuusi2015} and how they can be used for more general nonlinear operators. 
As it has been explained in \cite{kuusi2015}, the main difficulty in proving \eqref{eq:LevelSetEstimate} is that the reverse H\"older inequality \eqref{eq:RevHolderFinal} only holds on diagonal balls of the type $\cB(x_0,x_0,R)$. Thus maximal function arguments cannot be used, and we must resort to more direct arguments. 
We use a Calder\'on-Zygmund decomposition to decompose the level set $\{ H > \lambda \}$ into dyadic cubes. These cubes are then sorted into cubes situated on or near the diagonal (called ``diagonal" cubes) and cubes far from the diagonal (called ``off-diagonal" cubes). What is meant by ``far from" will be quantified below. The level set estimate for the diagonal cubes are handled using the reverse H\"older inequality \eqref{eq:RevHolderFinal}. It turns out that Sobolev functions automatically satisfy a type of reverse H\"older inequality on off-diagonal cubes, and we use this to obtain the level set estimate for said cubes.


%

\subsection{Vitali Covering}

Just as in \cite{kuusi2015}, we begin with an exit-time argument. The goal is to cover the portion of the set $\{ H > \lambda \}$ that lies on or near the diagonal $\{(x,x) \, : \, x \in \bbR^n  \} $. Let $\kappa \in (0,1]$ be a constant that will be chosen later, in \eqref{eq:KappaChoice}; all arguments in the paper up until then are independent of the choice of $\kappa$. Define
\begin{equation}\label{eq:LevelSetProof:1}
\lambda_1 := \frac{1}{\kappa} \sup_{\frac{\alpha-\beta}{40^n} \leq R \leq \frac{\varrho_0}{2}} \sup_{x \in B(x_0,\beta)} \{ \Psi_M(x,R) + \Upsilon_0(x,R) + Tail(x,R) \} \,.
\end{equation}
For the same $\kappa$ and for $\lambda \geq \lambda_1$, define the ``diagonal" level set of the functional $\Psi_M$ by
\begin{equation}\label{eq:LevelSetProof:2}
D_{\kappa \lambda} := \left\{ (x,x) \in \cB(x_0,\beta) \, : \, \sup_{0 < R < \frac{\alpha-\beta}{40^n}}  \Psi_M(x,R) > \kappa \lambda \right\}\,.
\end{equation}
Then by definition of $\lambda_1$ we have
\begin{equation}\label{eq:LevelSetProof:3}
\Psi_M(x,R) \leq \kappa \lambda_1 \leq \kappa \lambda\,, \qquad \text{ for each } (x,x) \in \cB(x_0,\beta)\,, \quad R \in \left[ \frac{\alpha-\beta}{40^n}, \frac{\varrho_0}{2} \right]\,,
\end{equation}
and so it follows that for every $(x,x)$ in the diagonal level set $D_{\kappa \lambda}$ there exists an exit time $R(x) \in (0, \frac{\alpha-\beta}{40^n})$ such that 
\begin{equation}\label{eq:LevelSetProof:4}
\Psi_M(x,R(x)) \geq \kappa \lambda \qquad \text{ while at the same time } \sup_{R(x) < R < \frac{\alpha-\beta}{40^n}} \Psi_M(x,R) \leq \kappa \lambda\,.
\end{equation}
Thus the collection $\{ \cB(x,2R(x)) \}$ forms a cover of $D_{\kappa \lambda}$, so by the Vitali covering theorem we can find a countable subcollection $\{ \cB(x_j, 2R(x_j)) \}$ such that
\begin{equation}\label{eq:LevelSetProof:5}
\bigcup_{(x,x) \in D_{\kappa \lambda}} \cB(x,2 R(x)) \subset \bigcup_{j} \cB(x_j,10 R(x_j))\,, \qquad  \cB(x_j,2R(x_j))  \text{ are mutually disjoint }\,.
\end{equation}
We hereafter use the abbreviations
\begin{equation}\label{eq:LevelSetProof:6}
\cB_j := \cB(x_j,R(x_j))\,, \qquad \sigma \cB_j := \cB(x_j, \sigma R(x_j))\,, \quad \sigma > 0\,.
\end{equation}
The quantity $\sum \nu(\cB_j)$ is treated by the diagonal estimates in the next section.
Note that since $x_j \in \cB(x_0,\beta)$ and $R(x_j) \leq \frac{\alpha-\beta}{40^n}$ we have $10 \cB_j \subset \cB(x_0,\alpha)$ for every $j$. By \eqref{eq:LevelSetProof:4} and by the doubling property of the measure $\nu$ in \eqref{eq:MeasureDoublingProperty} we also have
\begin{equation}\label{eq:LevelSetProof:7}
\sum_{j} \intdm{10\cB_j}{H^{p'}}{\nu} \leq \sum_j \nu( 10\cB_j) [\Psi_M(x_j, 10R(x_j))]^{p'} \leq 10^{n+\veps p} \kappa^{p'} \lambda^{p'} \sum_j \nu( \cB_j)\,.
\end{equation}

\subsection{Analysis On the Diagonal}
By \eqref{eq:LevelSetProof:4} it follows that at least one of two inequalities hold: either
\begin{equation}\label{eq:LevelSetProof:8}
\left( \fint_{\cB_j} H^{p'} \, \rmd \nu \right)^{1/p} \geq \frac{\kappa \lambda}{2}
\end{equation}
or
\begin{equation}\label{eq:LevelSetProof:9}
\frac{M [\nu(\cB_j)]^{\theta}}{\veps^{1/p_* - 1/p'}} \left( \fint_{\cB_j} F^{p_*} \, \rmd \nu \right)^{1/p_*} \geq \frac{\kappa \lambda}{2}\,.
\end{equation}
\underline{Case 1:} If \eqref{eq:LevelSetProof:8} occurs, then by the Reverse H\"older inequality \eqref{eq:RevHolderFinal}
\begin{equation}\label{eq:LevelSetProof:10}
\kappa \lambda \leq \frac{C}{\sigma \veps^{1/\gamma - 1/p'}} \left( \fint_{4 \cB_j} H^{\gamma} \, \rmd \nu \right)^{1/\gamma} + \frac{\sigma}{\veps^{1/\gamma - 1/p'}} \sum_{k=0}^{\infty} \alpha_k \left( \fint_{2^{k+2} \cB_j} H^{\gamma} \, \rmd \nu \right)^{1/\gamma} + \frac{C [\nu(\cB_j)]^{\theta}}{\veps^{1/p_* - 1/p'}} \left( \fint_{4 \cB_j} F^{p_*} \, \rmd \nu \right)^{1/p_*} \,,
\end{equation}
where $\sigma \in (0,1]$ has yet to be chosen and $C = C(\texttt{data})$. Choose the unique $m \in \bbZ_+$ such that $2^{-m} \varrho_0\leq R(x_j) < 2^{-m+1} \varrho_0$. Since $R(x_j) < \frac{\alpha-\beta}{40^n}$ and $0<\alpha-\beta<\varrho_0/2
$, we have $m \geq 3$. Further, $\frac{\alpha-\beta}{40^n} \leq \frac{\varrho_0}{2 \cdot 40^n} \leq 2^{m-1} R(x_j)$, so by \eqref{eq:LevelSetProof:1}
\begin{equation}\label{eq:LevelSetProof:11}
\Upsilon_0(x_j,2^{m-1}R(x_j)) + Tail(x_j,2^{m-1}R(x_j)) \leq \kappa \lambda_1\,.
\end{equation}
This allows us to estimate $Tail$. The first $m-2$ terms can be handled by the exit-time condition \eqref{eq:LevelSetProof:4}; that is,
\begin{equation}\label{eq:LevelSetProof:12}
\left( \fint_{2^k \cB_j} H^{\gamma} \, \rmd \nu  \right)^{1/\gamma}  \leq \kappa \lambda \quad \text{ if } 1 \leq k \leq m-2\,.
\end{equation}
Then by \eqref{eq:LevelSetProof:11} and \eqref{eq:LevelSetProof:12}, and recalling that $\alpha_k = 2^{-k(\frac{tq}{p-1} - s - \veps)}$,
\begin{equation}\label{eq:LevelSetProof:13}
\begin{split}
\sum_{k=0}^{\infty} \alpha_k \left( \fint_{2^{k+2}\cB_j} H^{\gamma} \, \rmd \nu \right)^{1/\gamma} &= \frac{1}{\alpha_2} \sum_{k=2}^{\infty} \alpha_k \left( \fint_{2^{k}\cB_j} H^{\gamma} \, \rmd \nu \right)^{1/\gamma}  \\
	&= \frac{1}{\alpha_2} \sum_{k=2}^{m-2} \alpha_k \left( \fint_{2^{k}\cB_j} H^{\gamma} \, \rmd \nu \right)^{1/\gamma} + \frac{1}{\alpha_2} \sum_{k=0}^{\infty} \alpha_{k+m-1} \left( \fint_{2^{k+m-1}\cB_j} H^{\gamma} \, \rmd \nu \right)^{1/\gamma} \\
	&\leq \frac{1}{\alpha_2} \left[ \kappa \lambda \sum_{k=2}^{m-2} \alpha_k + \alpha_{m-1} Tail(x_j, 2^{m-1} R(x_j)) \right] \\
	&\leq \frac{1}{\alpha_2} \left[ \kappa \lambda \sum_{k=2}^{m-2} \alpha_k + \alpha_{m-1} \kappa \lambda_1 \right] \\
	&\leq \kappa \lambda \sum_{k=0}^{\infty} \alpha_k   \leq \frac{4^{\frac{tq}{p-1}-s-\veps} \kappa \lambda}{(\frac{tq}{p-1} - s - \veps) \ln(2) } \leq \frac{4^q \kappa \lambda}{s(p'-1)(\frac{tq}{sp} - \frac{1}{p'} ) \ln(2) } := C_1 \kappa \lambda\,,
\end{split}
\end{equation}
where in the last line we used \eqref{eq:GeometricSeriesEstimate} and the bound $ \veps < s (\frac{tq}{sp} - \frac{1}{p'} )$ in \eqref{Assumption:Epsilon}. The constant $C_1$ depends only on $\texttt{data}$.
%
Now using the fact that $m \geq 3$ we gain that $2 R(x_j) \leq \frac{1}{2} \varrho_0$, so by the exit-time condition \eqref{eq:LevelSetProof:4}
\begin{equation}\label{eq:LevelSetProof:14}
\frac{C [\nu(\cB_j)]^{\theta}}{\veps^{1/p_*-1/p'}} \left( \fint_{4 \cB_j} F^{p_*} \, \rmd \nu \right)^{1/p_*} \leq C \frac{\Psi_M(x_j,2 R(x_j))}{M} \leq \frac{C_2 \kappa \lambda}{M}\,, \quad C_2 = C_2(\texttt{data})\,.
\end{equation}
Combining \eqref{eq:LevelSetProof:13} and \eqref{eq:LevelSetProof:14} in \eqref{eq:LevelSetProof:10} gives
\begin{equation}\label{eq:LevelSetProof:15}
\kappa \lambda \leq \frac{C}{\sigma \veps^{1/\gamma - 1/p'} } \left( \fint_{4 \cB_j} H^{\gamma} \, \rmd \nu \right)^{1/\gamma} + \frac{C_1 \sigma \kappa \lambda}{\veps^{1/\gamma - 1/p'}}  + \frac{C_2 \kappa \lambda}{M}\,.
\end{equation}
Now we choose $\sigma \in (0,1)$ and $M \geq 1$, and absorb the last two terms. We set
\begin{equation}\label{eq:LevelSetProof:16}
\sigma := \frac{\veps^{1/\gamma -  1/p'}}{4C_1}\,, \qquad M := 4 C_2 \,,
\end{equation}
and so we have for $C=C(\texttt{data})$
\begin{equation}\label{eq:LevelSetProof:17}
\kappa \lambda \leq \frac{C}{\veps^{2/\gamma - 2/p'}} \left( \fint_{4\cB_j} H^{\gamma} \, \rmd \nu \right)^{1/\gamma} \quad \Rightarrow \quad \nu(\cB_j) \leq \frac{C}{\veps^{2- (2\gamma)/p'} (\kappa \lambda)^{\gamma}} \int_{4 \cB_j} H^{\gamma} \, \rmd \nu\,.
\end{equation}
Now, let $\widetilde{\kappa} > 0$ be a constant, to be fixed in a moment. Using the doubling property \eqref{eq:MeasureDoublingProperty},
\begin{equation}\label{eq:LevelSetProof:18}
\begin{split}
\frac{C}{\veps^{2- (2\gamma)/p'} (\kappa \lambda)^{\gamma}} \int_{4 \cB_j} H^{\gamma} \, \rmd \nu 
	&\leq \frac{C}{\veps^{2- (2\gamma)/p'} (\kappa \lambda)^{\gamma}} \int_{4 \cB_j \cap \{ H \leq \widetilde{\kappa} \kappa \lambda \} } H^{\gamma} \, \rmd \nu  + \frac{C}{\veps^{2- (2\gamma)/p'} (\kappa \lambda)^{\gamma}} \int_{4 \cB_j \cap \{ H > \widetilde{\kappa} \kappa \lambda \} } H^{\gamma} \, \rmd \nu \\
	&\leq \frac{C_3 \nu(\cB_j) \widetilde{\kappa}^{\gamma}}{\veps^{2- (2\gamma)/p'} } + \frac{C_3}{\veps^{2- (2\gamma)/p'} (\kappa \lambda)^{\gamma}} \int_{4 \cB_j \cap \{ H > \widetilde{\kappa} \kappa \lambda \} } H^{\gamma} \, \rmd \nu\,,
\end{split}
\end{equation}
where $C_3 = C_3(\texttt{data})$. Choose
\begin{equation}\label{eq:LevelSetProof:19}
\widetilde{\kappa} = \frac{\veps^{2/\gamma -2/p'}}{(2 C_3)^{1/\gamma}}\,,
\end{equation}
and then substituting \eqref{eq:LevelSetProof:18} into \eqref{eq:LevelSetProof:17}$_2$ and absorbing the term gives
\begin{equation}\label{eq:LevelSetProof:20}
\nu(\cB_j) \leq \frac{C_4}{\veps^{2 - 2\gamma/p'} \kappa^{\gamma} \lambda^{\gamma}} \int_{4 \cB_j \cap \{ \widetilde{\kappa} \kappa \lambda \} } H^{\gamma} \, \rmd \nu\,, \qquad C_4 = C_4(\texttt{data})\,.
\end{equation}
\underline{Case 2:} If \eqref{eq:LevelSetProof:9} occurs then
\begin{equation*}
\left( \frac{\kappa \lambda}{2} \right)^{p_*} \leq \frac{M^{p_*} [\nu(\cB_j) ]^{ p_* \theta - 1}}{\veps^{1-p_* / p'}} \intdm{\cB_j}{F^{p_*}}{\nu}\,,
\end{equation*}
which implies
\begin{equation}\label{eq:LevelSetProof:21}
\nu(\cB_j) \leq \left( \frac{2M}{\veps^{1/p_*-1/p'} \kappa \lambda} \right)^{p_*/(1-p_* \theta)} \left( \intdm{\cB_j}{F^{p_*}}{\nu} \right)^{1/(1-p_* \theta)}\,
\end{equation}
where we used the inequality $p_* \theta < 1$ which follows from the definition of $\theta$ that $p_* \theta \leq \frac{nsp}{(n+sp)(n+\veps p)} < \frac{p}{p+1}$. 
Moreover, since $p\geq 2,$ we have that $3\leq {1\over 1-p_* \theta} = 1+ {p_* \theta\over 1-p_* \theta}$. This decomposition of the exponent will allow us to  incorporate the level set of $F$ in the integrand, and then remove the exponent on the integral. To that end, with a constant $\wh{\kappa} \in (0,1)$ to be determined, we split the integral as
\begin{equation}\label{eq:LevelSetProof:22}
\left( \intdm{\cB_j}{F^{p_*}}{\nu} \right)^{{1 \over 1-p_* \theta }} \leq \left( \intdm{\cB_j \cap \{ F > \wh{\kappa} \kappa \lambda \} }{F^{p_*}}{\nu} + (\wh{\kappa} \kappa \lambda )^{p_*} \nu(\cB_j) \right)^{{1 \over 1-p_* \theta }}\,.
\end{equation}
Now, since $\varrho_0 \leq 1$, 
\begin{equation*}\label{eq:AddTerms:DefnOfH}
\nu(\cB(x_0,2 \varrho_0) \leq \frac{C(n) 2^{n+\veps p}}{\veps} := L = L(n,p,\veps)\,.
\end{equation*}
Then by noting that $\cB_j \subset \cB(x_0,2 \varrho_0)$ we can estimate 
\begin{equation*}\label{eq:LevelSetProof:23}
[\nu(\cB_j)]^{{1 \over 1-p_* \theta }} \leq [\nu(\cB(x_0,2\varrho_0))]^{{p_* \theta \over 1-p_* \theta }}  \nu(\cB_j) \leq L^{{p_* \theta \over 1-p_* \theta }} \nu(\cB_j)\,.
\end{equation*}
Then using the elementary inequality $(a+b)^r \leq 2^{r-1}(a^r + b^r)$ for any $r \geq 1$ in \eqref{eq:LevelSetProof:22} and using the estimate above for $[\nu(\cB_j)]^{1/(1-p_* \theta )} $ we have 
\begin{equation}\label{eq:LevelSetProof:24}
\begin{split}
\left( \intdm{\cB_j}{F^{p_*}}{\nu} \right)^{{1\over 1-p_* \theta }} &\leq 2^{p_* \theta  / (1- p_*\theta ) }\left( \intdm{\cB_j \cap \{ F > \wh{\kappa} \kappa \lambda \} }{F^{p_*}}{\nu} \right)^{{1\over 1-p_* \theta }}\\
& + ( 2(L+1) )^{{p_* \theta \over 1-p_* \theta }} \left( \wh{\kappa} \kappa \lambda \right)^{{p_* \theta \over 1-p_* \theta }} \nu(\cB_j)\,.
\end{split}
\end{equation}
Combining \eqref{eq:LevelSetProof:21} and \eqref{eq:LevelSetProof:24}, and using that $\theta \leq 1$,
\begin{equation*}
\nu(\cB_j) \leq \left( \frac{2^2 M}{\veps^{1/p_* - 1/p'} \kappa \lambda} \right)^{p_*/(1 - p_* \theta )}  \left( \intdm{\cB_j \cap \{ F > \wh{\kappa} \kappa \lambda \} }{F^{p_*}}{\nu} \right)^{1/(1-p_* \theta )} + \left( \frac{2^{2} M (L+1) \wh{\kappa}}{\veps^{1/p_* - 1/p'}} \right)^{p_*/(1-\theta p_*)} \nu(\cB_j)\,.
\end{equation*}
Now set $\wh{\kappa} \in (0,1)$ to satisfy
\begin{equation}\label{eq:LevelSetProof:25}
\left( \frac{4 M (L+1) \wh{\kappa}}{\veps^{1/p_* - 1/p'}} \right)^{p_*/(1-  p_* \theta)} \leq  \frac{1}{2} \quad \Rightarrow \quad \wh{\kappa} \leq \left( \frac{1}{2} \right)^{(1-p_* \theta )/p_*} \frac{\veps^{1/p_* - 1/p'}}{4 M (L+1)}\,.
\end{equation}
Then
\begin{equation*}
\nu(\cB_j) \leq 2 \left( \frac{4 M}{\veps^{1/p_* - 1/p'} \kappa \lambda} \right)^{{p_*\over 1-p_* \theta }}  \left( \intdm{\cB_j \cap \{ F > \wh{\kappa} \kappa \lambda \} }{F^{p_*}}{\nu} \right)^{{1\over 1-p_* \theta }}\,.
\end{equation*}
Writing $\left( \intdm{\cB_j \cap \{ F > \wh{\kappa} \kappa \lambda \} }{F^{p_*}}{\nu} \right)^{{1\over 1-p_* \theta }} = \left( \intdm{\cB_j \cap \{ F > \wh{\kappa} \kappa \lambda \} }{F^{p_*}}{\nu} \right) \left( \intdm{\cB_j \cap \{ F > \wh{\kappa} \kappa \lambda \} }{F^{p_*}}{\nu} \right)^{{p_* \theta \over 1-p_* \theta }}$ let us transfer some of the ``decay" from the integrand to the cutoff $\lambda_0$, and in so doing remove the exponent from the integral.  
\begin{equation*}
\begin{split}
\intdm{\cB_j \cap \{ F > \wh{\kappa} \kappa \lambda \} }{F^{p_*}}{\nu} &\leq  (\wh{\kappa} \kappa \lambda)^{p_*} \intdm{\cB_j \cap \{ F > \wh{\kappa} \kappa \lambda \} }{\left( \frac{F}{\wh{\kappa} \kappa \lambda} \right)^{p_* + \delta_f}}{\nu} \\
	&\leq \frac{\nu(2^{m-1}\cB_j)}{(\wh{\kappa} \kappa \lambda)^{\delta_f}} \fint_{2^{m-1}\cB_j} F^{p_*+\delta_f} \, \rmd \nu \\
	&\leq \frac{\nu(\cB(x_0,2\varrho_0))}{(\wh{\kappa} \kappa \lambda)^{\delta_f}} [\Upsilon_0(x_j,2^{m-1}R(x_j))]^{p_*+\delta_f} \leq \frac{L \lambda_1^{p_*+\delta_f}}{(\wh{\kappa} \kappa \lambda)^{\delta_f}}\,.
\end{split}
\end{equation*}
Therefore, powering the above inequality by ${p_* \theta \over 1-p_* \theta }$ on both sides we obtain that 
\begin{equation}\label{eq:LevelSetProof:26:a}
\nu(\cB_j) \leq \frac{C_{5} \lambda_1^{(p_*+\delta_f)\theta p_* /(1-p_* \theta )}}{(\wh{\kappa} \kappa \lambda)^{(1+\theta \delta_f) p_* / (1- p_* \theta )}} \intdm{\cB_j \cap \{ F > \wh{\kappa} \kappa \lambda \} }{F^{p_*}}{\nu}\,,
\end{equation}
where
\begin{equation}\label{eq:LevelSetProof:27}
C_{5} := 2 \left( \frac{4 M (L+1)}{\veps^{1/p_* - 1/p'}} \right)^{p_*/(1-p_* \theta)}\,.
\end{equation}
Combining \eqref{eq:LevelSetProof:20} and \eqref{eq:LevelSetProof:26:a},
\begin{equation*}
\nu(\cB_j) \leq \frac{C_4}{\veps^{2-2\gamma/p'} \kappa^{\gamma} \lambda^{\gamma}} \intdm{4 \cB_j \cap \{ H > \widetilde{\kappa} \kappa \lambda \} }{H^{\gamma}}{\nu}
	+ \frac{C_5 \lambda_1^{\vartheta_f}}{(\wh{\kappa} \kappa \lambda)^{\wt{\vartheta}_f}} \intdm{\cB_j \cap \{ F > \wh{\kappa} \kappa \lambda \} }{F^{p_*}}{\nu}\,,
\end{equation*}
and since $\{ \cB_j \}$ is a disjoint collection whose members are all contained in $\cB(x_0,\alpha)$ we have
\begin{equation}\label{eq:LevelSetProof:28}
\sum_j \nu(\cB_j) \leq \frac{C_4}{\veps^{2-2 \gamma/p'} \kappa^{\gamma} \lambda^{\gamma}} \intdm{ \cB(x_0,\alpha) \cap \{ H > \widetilde{\kappa} \kappa \lambda \} }{H^{\gamma}}{\nu}
	+ \frac{C_5 \lambda_1^{\vartheta_f}}{(\wh{\kappa} \kappa \lambda)^{\widetilde{\vartheta}_f}} \intdm{\cB(x_0,\alpha) \cap \{ F > \wh{\kappa} \kappa \lambda \} }{F^{p_*}}{\nu}\,,
\end{equation}
where the constants $C_4 = C_4(\texttt{data})$, $C_5(\texttt{data},\veps)$, $\widetilde{\kappa}$ and $\wh{\kappa}$ have been chosen to satisfy \eqref{eq:LevelSetProof:20}, \eqref{eq:LevelSetProof:27}, \eqref{eq:LevelSetProof:19} and \eqref{eq:LevelSetProof:25} respectively. The constant $\kappa \in (0,1]$ defined in \eqref{eq:KappaChoice} will be fixed in the course of the off-diagonal estimates below.

\subsection{Analysis Off the Diagonal}\label{subsec:OffDiag}
The analysis far from the diagonal is much more technical. We begin with defining the collections of dyadic cubes and summarizing their properties that will be used repeatedly. We then define two dimensional constants that will be used throughout the analysis of cubes far from the diagonal. In Section \ref{subsubsec:CZDecomp} we recall the classical Calder\'on-Zygmund decomposition and adapt it to decompose the level set $\{ H > \lambda \}$ into the aforementioned dyadic cubes. In Section \ref{subsubsec:NearDiag} we analyze the cubes near the diagonal and show they can be covered by the collection $\{ \cB_j \}$. In the remaining sections we treat the cubes far from the diagonal using the ``almost-reverse" H\"older inequality in Lemma \ref{lma:OffDiagRevHolderIneq}, a careful choice of the constant $\kappa$, and combinatorial information about the cubes coming from their size and distance from the diagonal.

\subsubsection{Dyadic Cubes}
The following contains information regarding cubes arising from a Calder\'on-Zygmund decomposition of Euclidean space. The dyadic cubes considered here are centered at $x_0$ and their size has been changed so as to be compatible with the starting ball $\cB(x_0,\beta)$. We will consider cubes with side length $2^{-k}$ for integers $k \geq k_0$, where
\begin{equation}\label{eq:MinSideLength}
k_0 := \left\lfloor -\log_2 \left( \frac{\alpha-\beta}{n 40^{n+1}} \right) \right\rfloor + 1\,.
\end{equation}
Here, $\lfloor \cdot \rfloor$ denotes the floor function. For each $k \geq k_0$, let $\cC_k$ be the disjoint collection of half-open cubes in $\bbR^n$ centered at $x_0$ with side length $2^{-k}$ whose closures intersect $\overline{B}(x_0,\frac{1}{2}(\alpha-\beta))$.
That is,
\begin{equation*}
\cC_k := \{ x_0 + 2^{-k}z + [0,2^{-k})^n \, : z \in \bbZ^n\,, \, \, (x_0 + 2^{-k}z + [0,2^{-k})^n) \cap \overline{B}(x_0,\frac{1}{2}(\alpha-\beta)) \neq \emptyset \}\,.
\end{equation*}
Then by \eqref{eq:MinSideLength} we have 
\begin{equation}\label{eq:LevelSetProof:CubesAreCovering}
B(x_0,\beta) \subset \bigcup_{K \in \cC_k} K \subset B(x_0,\alpha)\,.
\end{equation}
Note that every cube $K \in \cC_{k+1}$ has a unique ``predecessor" $\widetilde{K} \in \cC_k$ such that $K \subset \widetilde{K}$. With these cubes, we can define dyadic cubes on $\bbR^{2n}$, denoted by
\begin{equation}\label{eq:DefnOfCubeCollection}
\Delta_k := \{ \cK := K_1 \times K_2 \, : \, K_1, K_2 \in \cC_k \}\,, \qquad \Delta := \bigcup_{k \geq k_0} \Delta_k\,.
\end{equation}
Denote the diagonal cubes $\Delta_k^d := \{ K \times K \, : \, K \in \Delta_k \}$. Then
\begin{equation*}
\cB(x_0,\beta) \subset \bigcup_{\cK \in \Delta_k} \cK \subset \cB(x_0,\alpha)\,.
\end{equation*}
Note that these product cubes in $\bbR^{2n}$ satisfy all of the same properties as the cubes. Numerous times a cube $\cK \in \Delta$ will be given; there exists $K_1$, $K_2 \in \cC_k$ such that $\cK = K_1 \times K_2$, and we denote
\begin{equation*}
k(\cK) = k\,.
\end{equation*}
We denote the cube projections for a cube $\cK = K_1 \times K_2$ by
\begin{equation*}
\Kone := K_1 \times K_1\,, \qquad \text{ and } \qquad \Ktwo := K_2 \times K_2\,.
\end{equation*}

\begin{proposition}\label{prop:Cubes}
Let $\cK = K_1 \times K_2 \in \Delta$. The following hold:
\begin{itemize}
\item $\Kone$, $\Ktwo \in \Delta$.

\item $\nu(\Kone) = \nu(\Ktwo)$ and $k(\cK) = k(\Kone) = k(\Ktwo)$.

\item If $\cH \in \Delta$ and $\cH \subset \cK$, then $k(\cK) \leq k(\cH)$.

\item If $\widetilde{\cK} = \widetilde{K}_1 \times \widetilde{K}_2$ is the predecessor of $\cK$, then 
\begin{equation}\label{eq:MonotonicityOfCubes}
dist (\widetilde{K}_1, \widetilde{K}_2) \leq \dist(K_1, K_2)\,.
\end{equation}

\item The following hold:
\begin{equation}
\dist(K_1,K_2) = \frac{1}{\sqrt{2}} \dist(\Kone,\Ktwo)\,.
\end{equation}
\begin{equation}\label{eq:DistBetweenCubeAndDiag}
\dist(\cK,\mathrm{diag}) = \frac{1}{\sqrt{2}} \dist(K_1,K_2) = \frac{1}{2} \dist(\Kone,\Ktwo) \,.
\end{equation}
Note that for two sets $\cA_1$ and $\cA_2 \subset \bbR^{2n}$ the definition of $\dist(\cA_1, \cA_2)$ uses the traditional Euclidean norm.
\end{itemize}
\end{proposition}

\subsubsection{Two Dimensional Constants} In the analysis that follows we will often compare the sizes of the cubes with their distance to the diagonal. This leads us to define two dimensional constants independent of $\veps$ that will be used repeatedly. The existence of these constants follows from geometric arguments.

\begin{lemma}\label{lma:TwoDimensionalConstants}
There exists a constant $C_{dd}$ depending only on $n$ such that
for $h \in \{1,2\}$
\begin{equation}\label{eq:LevelSetProof:Cdd}
C_{dd} \geq \sup_{\cK \in \Delta} \left\{ \frac{1}{\veps} \left( \frac{\dist(K_1,K_2)}{2^{-k}} \right)^{n - \veps p} \frac{\nu(\cK)}{\nu(\Kh)} \right\} + \sup_{\substack{\cK \in \Delta \\ \dist(K_1, K_2) \geq 2^{-k}}} \left\{ \veps \left( \frac{\dist(K_1,K_2)}{2^{-k}} \right)^{\veps p - n} \frac{\nu(\Kh)}{\nu(\cK)} \right\}\,,
\end{equation}
where $k = k(\cK)$.
\end{lemma}
\begin{proof}
Without loss of generality assume $\dist(K_1,K_2) > 0$. By definition of $\nu$
\begin{equation*}
\nu(\Kh) \geq \frac{C(n) (2^{-k})^{n+ \veps p} }{\veps}\,, \qquad \text{while} \qquad \nu(\cK) \leq \frac{2^{(-k)2n}}{\dist(K_1,K_2)^{n-\veps p}}\,.
\end{equation*}
Thus the first quantity is bounded;
\begin{equation*}
\frac{1}{\veps} \left( \frac{\dist(K_1,K_2)}{2^{-k}} \right)^{n - \veps p} \frac{\nu(\cK)}{\nu(\Kh)} \leq C(n)\,.
\end{equation*}
On the other hand, if $x \in K_1$ and $y \in K_2$ then
\begin{equation*}
\begin{split}
\dist(K_1,K_2) &\leq |x-y| \leq \dist(K_1,K_2) + 2\sqrt{n} (2^{-k}) \\
	&\leq 2 \sqrt{n} (\dist(K_1,K_2) + 2^{-k} )\,.
\end{split}
\end{equation*}
Thus by definition of $\nu$
\begin{equation*}
\nu(\cK) \geq \frac{(2^{-k})^{2n}}{(2\sqrt{n})^{n-\veps p} ( \dist(K_1,K_2) + 2^{-k} )^{n-\veps p} }\,.
\end{equation*}
Since
\begin{equation*}
\nu(\Kh) = \frac{C(n)}{\veps} (2^{-k})^{n+\veps p}\,,
\end{equation*}
we have
\begin{equation*}
\begin{split}
\veps \left( \frac{\dist(K_1,K_2)}{2^{-k}} \right)^{\veps p - n} \frac{\nu(\Kh)}{\nu(\cK)} &\leq c(n) \left( \frac{\dist(K_1,K_2)}{2^{-k}} \right)^{\veps p - n} \frac{(\dist(K_1,K_2) + 2^{-k} )^{n-\veps p}}{(2^{-k})^{n-\veps p}} \\
	&\leq c(n) \left( \frac{\dist(K_1,K_2)}{2^{-k}} \right)^{\veps p' - n} \frac{(\dist(K_1,K_2) )^{n-\veps p}}{(2^{-k})^{n-\veps p}} \leq C(n)\,,
\end{split}
\end{equation*}
where in the second inequality we used that $\dist(K_1,K_2) \geq 2^{-k}$. Thus the second term is bounded by a dimensional constant as well, and so \eqref{eq:LevelSetProof:Cdd} holds for some constant $C_{dd}$ depending only on $n$.
\end{proof}

\begin{lemma}\label{lma:TwoDimensionalConstants2}
There exists a constant $C_{ddd}$ depending only on $n$ such that
\begin{equation}\label{eq:LevelSetProof:Cddd}
\sup \left\{ \frac{\nu(\widetilde{\cK})}{\nu(\cK)} \, : \, \widetilde{\cK} \text{ is the predecessor of } \cK\,, \, \dist(\widetilde{K_1},\widetilde{K_2}) \geq 2^{-k(\cK)} \right\} \leq C_{ddd}\,.
\end{equation}
\end{lemma}

\begin{proof}
The triangle inequality gives
\begin{equation}\label{eq:LevelSetProof:SecondConstantProof}
|x-y| \leq 2 \sqrt{n} 2^{-k(\cK) + 1} + \dist(\widetilde{K_1}, \widetilde{K_2}) \leq 8 \sqrt{n} \dist(\widetilde{K_1},\widetilde{K_2})\,,
\end{equation}
whenever $x \in K_1$, $y \in K_2$ and $\dist(\widetilde{K_1},\widetilde{K_2}) \geq 2^{-k(\cK)}$. Therefore,
\begin{equation*}
\begin{split}
\nu(\widetilde{\cK}) &\leq \dist(\widetilde{K_1},\widetilde{K_2})^{-n-\veps p} |\widetilde{K_1}| |\widetilde{K_2}| \\
	&= 4^n \dist(\widetilde{K_1},\widetilde{K_2})^{-n-\veps p} \iintdm{K_1}{K_2}{}{x}{y} \\
	&\EquationReference{\eqref{eq:LevelSetProof:SecondConstantProof}}{\leq} 4^n 8 \sqrt{n} \iintdm{K_1}{K_2}{\frac{1}{|x-y|^{n-\veps p}}}{x}{y} = C(n) \nu(\cK)\,.
\end{split}
\end{equation*}
\end{proof}

\subsubsection{Calder\'on-Zygmund Covering and Sorting of Dyadic Cubes}\label{subsubsec:CZDecomp}
Here we write down a version of the classical Calder\'on-Zygmund decomposition adapted for use in our context. The proof is virtually identical to the classical proof, taking into account that the measure is doubling and absolutely continuous with respect to Lebesgue measure; see \cite{kuusi2015, Stein}.

\begin{theorem}\label{thm:CZDecomp}
Let $Q_0 \subset \bbR^{2n}$ be a cube and let $V \geq 0$ be a function in $L^1(Q_0)$. Let $\widetilde{\lambda}$ be a number such that
\begin{equation*}
\fint_{Q_0} V \, \rmd \nu \leq \widetilde{\lambda}\,.
\end{equation*}
Then there exists a collection of at most countable cubes $\{ Q_i \}$ that are pairwise disjoint with sides parallel to those of $Q_0$ such that
\begin{equation*}
\widetilde{\lambda} < \fint_{Q_i} V \, \rmd \nu \qquad \text{ and } \qquad \fint_{\widetilde{Q}_i} V \, \rmd \nu \leq \widetilde{\lambda} \quad \text{ for every } Q_i\,,
\end{equation*}
where $\widetilde{Q}_i$ is the predecessor of $Q_i$, and\begin{equation*}
V \leq \widetilde{\lambda} \quad \text{ almost everywhere in } Q_0 \setminus \bigcup_i Q_i\,.
\end{equation*}
\end{theorem}

We will use this theorem to cover the level set of $H$. Define
\begin{equation}\label{eq:DefnOfLambda2}
\lambda_2 := \max \left\{ \lambda_1, \sup_{\cK \in \Delta_{k_0}} \left( \fint_{\cK} H^{p'} \, \rmd \nu \right)^{1/p'} \right\}\,.
\end{equation}
Recall that $\lambda_1$ and $k_0$ have been defined in \eqref{eq:LevelSetProof:1} and \eqref{eq:MinSideLength} respectively. Note that by \eqref{eq:LevelSetProof:CubesAreCovering} the cubes $\{ \cK \}_{\cK \in \Delta_{k_0}}$ cover $\cB(x_0,\beta)$. For $\lambda \geq \lambda_2$ apply Theorem \ref{thm:CZDecomp} with $Q_0 = \cK_0$ for each and every cube $\cK_0 \in \Delta_{k_0}$. We obtain a pairwise disjoint family of cubes $Q_i (\cK_0)$ such that
\begin{equation}\label{eq:CZSplittingForH}
\lambda^{p'} < \fint_{Q_i(\cK_0)} H^{p'} \, \rmd \nu \qquad \text{ and } \qquad \fint_{\widetilde{Q}_i(\cK_0)} H^{p'} \, \rmd \nu \leq \widetilde{\lambda}^{p'} \quad \text{ for every } Q_i(\cK_0)\,,
\end{equation}
where $\widetilde{Q}_i(\cK_0)$ denotes the predecessor of $Q_i(\cK_0)$, and 
\begin{equation*}
H \leq \lambda \quad \text{a.e. in } \cK_0 \setminus \bigcup_i Q_i(\cK_0)\,.
\end{equation*}
Thus, we get a countable collection of disjoint dyadic cubes
\begin{equation*}
\Hlambda := \bigcup_{\cK_0 \in \Delta_{k_0}} \{ Q_i(\cK_0) \} = \{\cK \} 
\end{equation*}
that satisfy
\begin{equation}\label{eq:CZBoundForH}
\lambda^{p'} < \fint_{\cK} H^{p'} \, \rmd \nu \quad \text{ and } \quad \fint_{\widetilde{\cK}} H^{p'} \, \rmd \nu \leq \lambda^{p'} \qquad \text{ for every } \cK \in \Hlambda
\end{equation}
where $\widetilde{\cK}$ denotes the predecessor of $\cK$ and such that
\begin{equation}\label{eq:BoundOnHOutsideOfHLambda}
H \leq \lambda \quad \text{a.e. in } \cB(x_0,\alpha) \setminus \bigcup_{\cK \in \Hlambda} \cK\,.
\end{equation}

\subsubsection{Nearly Diagonal Cubes}\label{subsubsec:NearDiag}
It turns out that the cubes whose distance to the diagonal is smaller than their size can be covered by diagonal balls $\cB_j$ chosen in the exit-time argument above, see \eqref{eq:LevelSetProof:4}-\eqref{eq:LevelSetProof:5}. This leaves us to deal with the cubes that are ``far" from the diagonal in the next section.

To begin, we define the family of ``nearly diagonal" cubes
\begin{equation*}
\Hdlambda := \{ \cK \in \Hlambda \, : \, \dist(\widetilde{K}_1, \widetilde{K}_2) < 2^{-k(\cK)}\,, \quad \widetilde{\cK} = \widetilde{K}_1 \times \widetilde{K}_2 \text{ is the predecessor of } \cK \}\,.
\end{equation*}
With $\cK \in \Hdlambda$, let $(\widetilde{x},\widetilde{x}) \in \Diag$ such that $\dist((\widetilde{x},\widetilde{x}),\widetilde{\cK}) = \dist(\Diag, \widetilde{\cK}) $ and a diagonal ball $\cB(\widetilde{x},R)$ such that
\begin{equation*}
R \geq \frac{5 \sqrt{n}}{2} \dist(\widetilde{K}_1, \widetilde{K}_2) + 5 \sqrt{n} 2^{-k(\cK) + 1}\,.
\end{equation*}
Using \eqref{eq:DistBetweenCubeAndDiag} for $\widetilde{\cK}$, it follows that $\widetilde{\cK} \subset \cB(\widetilde{x},R)$. Thus, the diagonal ball $\cB \equiv \cB(\widetilde{x}, 24 \sqrt{n} 2^{-k(\cK)})$ satisfies $\cB \subset \cK$. Now note that by \eqref{eq:MeasureOfBallAndCube} in Theorem \ref{thm:Measure} there exists a constant $C_d = C_d(n,p)$ such that
\begin{equation*}
1  \leq \frac{\nu(\cB)}{\nu(\cK)} \leq \frac{C_d}{\veps}\,. 
\end{equation*}
Thus if $\cK \in \Hdlambda$ then the lower bound in \eqref{eq:CZBoundForH} gives
\begin{equation*}
\lambda^{p'} < \fint_{\cK} H^{p'} \, \rmd \nu \leq \frac{\nu(\cB)}{\nu(\cK)} \fint_{\cB} H^{p'} \, \rmd \nu \leq \frac{C_d}{\veps} \fint_{\cB} H^{p'} \, \rmd \nu\,.
\end{equation*}
By choosing the number $\kappa \in (0,1]$ introduced in \eqref{eq:LevelSetProof:1} to satisfy
\begin{equation}\label{eq:KappaCondition1}
\kappa \in (0,\kappa_0]\,, \qquad \kappa_0 := \frac{\veps^{1/p'}}{(2 C_d)^{1/p'}}\,,
\end{equation}
we have therefore proved that
\begin{equation*}
\text{ For every } \cK \in \Hdlambda \text{ there exists } \cB^{\cK} = B^{\cK} \times B^{\cK} \text{ such that } \kappa^{p'} \lambda^{p'} < \fint_{\cB^{\cK}} H^{p'} \, \rmd \nu \text{ with } \cK \subset \cB^{\cK}\,.
\end{equation*}
Denote the center of $\cB^{\cK}$ by $\widetilde{x}$. Then by the choice \eqref{eq:MinSideLength} for the lower bound on $k(\cK)$, it follows that the radius of the ball $\cB \equiv \cB(\widetilde{x},24 \sqrt{n} 2^{-k(\cK)})$ is smaller than $\frac{\alpha-\beta}{40^n}$. Therefore we can apply the exit time condition  \eqref{eq:LevelSetProof:4} to obtain that $(\widetilde{x},\widetilde{x}) \in C_{\kappa \lambda}$ and then $\cB^{\cK} \subset \cB(\widetilde{x},R(\widetilde{x}))$. By \eqref{eq:LevelSetProof:5} it follows that
\begin{equation}\label{eq:DiagCubesContained}
\bigcup_{\cK \in \Hdlambda} \cK \subset \bigcup_j 10 \cB_j\,.
\end{equation}

\subsubsection{Off-Diagonal Reverse H\"older Inequalities}

Since the nearly diagonal cubes $\Hdlambda$ have been covered by the diagonal cover, we need only consider the off-diagonal cubes
\begin{equation}\label{eq:DefnOfHndlambda}
\Hndlambda := \{ \cK \in \Hlambda \, : \, \dist(\widetilde{K}_1, \widetilde{K}_2) \geq 2^{-k(\cK)}\,, \quad \widetilde{\cK} = \widetilde{K}_1 \times \widetilde{K}_2 \text{ is the predecessor of } \cK \}\,.
\end{equation}
By \eqref{eq:MonotonicityOfCubes} we also have
\begin{equation*}
\dist(K_1, K_2) \leq 2^{-k(\cK)} \qquad \text{ for every } \cK \in \Hndlambda\,.
\end{equation*}

Our objective now is to categorize and estimate sums of the measures of cubes in $\Hndlambda$. We will use the following lemma to do so. This lemma states that for off-diagonal cubes an ``almost reverse H\"older inequality" holds automatically regardless of whether the function $u$ solves an equation. However, diagonal cubes appear in the estimate, which must be treated by a combinatorial argument in subsequent sections.

\begin{lemma}\label{lma:OffDiagRevHolderIneq}
Let $k \geq k_0$ and let $\cK = K_1 \times K_2 \in \Delta_k$. Then there exists a constant $C_{nd} = C_{nd}(\texttt{data})$ independent of $\veps$ such that if $\dist(K_1,K_2) \geq 2^{-k}$ then
\begin{equation}\label{eq:OffDiagRevHolderIneq}
\begin{split}
\left( \fint_{\cK} H^{p'} \, \rmd \nu \right)^{1/p'} 
	&\leq C_{nd} \left( \fint_{\cK} H^{\gamma} \, \rmd \nu \right)^{1/\gamma} \\
	&\quad + \frac{C_{nd}}{\veps^{1/\gamma}} \left( \frac{2^{-k}}{\dist(K_1,K_2)} \right)^{(p-1)(s+\veps)} \left[ \left( \fint_{\Kone} H^{\gamma} \, \rmd \nu \right)^{1/\gamma} + \left( \fint_{\Ktwo} H^{\gamma} \, \rmd \nu \right)^{1/\gamma} \right]\,,
\end{split}
\end{equation}
with $\gamma$ defined as in Corollary \ref{cor:RevHolderFinal}. In particular, \eqref{eq:OffDiagRevHolderIneq} holds whenever $\cK \in \Hndlambda$.
\end{lemma}

\begin{proof}
First, there exist $x_1 \in \overline{K_1}$ and $y_1 \in \overline{K_2}$ such that $\dist(K_1,K_2) = |x_1 - y_1|$. Then for any $(x,y) \in \cK$,
\begin{equation*}
\begin{split}
|x-y| &\leq \dist(K_1, K_2) + |x_1 - x| + |y_1 - y| \\
	&\leq \dist(K_1, K_2) + 2 \sqrt{n} 2^{-k} \\
	&\leq 3 \sqrt{n} \dist(K_1, K_2)\,,
\end{split}
\end{equation*}
since $\dist(K_1,K_2) \geq 2^{-k}$. Therefore,
\begin{equation}\label{eq:OffDiagRevHolderIneq:Proof1}
1 \leq \frac{|x-y|}{\dist(K_1,K_2)}  \leq 3 \sqrt{n} \quad \text{ for all } (x,y) \in \cK\,;
\end{equation}
the first inequality is a consequence of the definition of $\dist(K_1, K_2)$. Next, by definition of $\nu$ we have
\begin{equation}\label{eq:OffDiagRevHolderIneq:Proof2}
\frac{1}{C(n,p)} \frac{4^{-nk}}{\dist(K_1,K_2)^{n-\veps p}} \leq \nu(\cK) \leq C(n,p) \frac{4^{-nk}}{\dist(K_1,K_2)^{n-\veps p}}\,.
\end{equation}
We therefore have, using \eqref{eq:OffDiagRevHolderIneq:Proof1} and \eqref{eq:OffDiagRevHolderIneq:Proof2},
\begin{equation}\label{eq:OffDiagRevHolderIneq:Proof3}
\begin{split}
\left( \fint_{\cK} H^{p'} \, \rmd \nu \right)^{1/p'} 
	&= \left( \frac{1}{\nu(\cK)} \iintdm{K_1}{K_2}{ \frac{|u(x)-u(y)|^p}{|x-y|^{n+sp}} + a(x,y) \frac{|u(x)-u(y)|^q}{|x-y|^{n+tq}} }{y}{x} \right)^{1/p'} \\
	&\leq C \bigg[ \left( \frac{\dist(K_1,K_2)^{n-\veps p - (n+sp)}}{4^{-nk}} \iintdm{K_1}{K_2}{|u(x)-u(y)|^p}{y}{x} \right) \\
	&\qquad \qquad + \left( \frac{\dist(K_1,K_2)^{n-\veps p - (n+tq)}}{4^{-nk}} \iintdm{K_1}{K_2}{|u(x)-u(y)|^q}{y}{x} \right)\bigg]^{1/p'} \\
	&\leq C \dist(K_1,K_2)^{-(p-1)(s+\veps)} \bigg[ \fint_{K_1} \fint_{K_2} {|u(x)-u(y)|^p} \, \rmd y \, \rmd x  \\
	&\qquad \qquad + \dist(K_1,K_2)^{sp-tq} \fint_{K_1} \fint_{K_2} {|u(x)-u(y)|^q} \, \rmd y \, \rmd x \bigg]^{1/p'}\,,
\end{split}
\end{equation}
where $C = C(\texttt{data})$. Using \eqref{eq:LevelSetProof:CubesAreCovering} we can estimate
\begin{equation*}
\dist(K_1,K_2)^{sp-tq} \leq (2 \alpha)^{sp-tq} \leq (3 \varrho_0)^{sp-tq} \leq 3^{sp-tq} \equiv C(\texttt{data})\,,
\end{equation*}
and we also have
\begin{equation*}
\begin{split}
\fint_{K_1} \fint_{K_2} {|u(x)-u(y)|^q} \, \rmd y \, \rmd x  
	&\leq (2 \Vnorm{u}_{L^{\infty}(\bbR^n)})^{q-p} \fint_{K_1} \fint_{K_2} {|u(x)-u(y)|^p} \, \rmd y \, \rmd x \\
	&\leq C(\texttt{data}) \fint_{K_1} \fint_{K_2} {|u(x)-u(y)|^p} \, \rmd y \, \rmd x\,.
\end{split}
\end{equation*}
Combining these two estimates with \eqref{eq:OffDiagRevHolderIneq:Proof3} yields
\begin{equation}\label{eq:OffDiagRevHolderIneq:Proof4}
\left( \fint_{\cK} H^{p'} \, \rmd \nu \right)^{1/p'} \leq C \dist(K_1,K_2)^{-(p-1)(s+\veps)} \left( \fint_{K_1} \fint_{K_2} {|u(x)-u(y)|^p} \, \rmd y \, \rmd x \right)^{1/p'}\,,
\end{equation}
where $C = C(\texttt{data})$. Now, using Minkowski's inequality,
\begin{equation*}
\begin{split}
\left( \fint_{K_1} \fint_{K_2} {|u(x)-u(y)|^p}  \, \rmd y \, \rmd x \right)^{1/p} &\leq \left( \fint_{K_1} {|u(x)-(u)_{K_1}|^p}  \, \rmd x \right)^{1/p} \\
&+ \left(\fint_{K_2} {|u(x)-(u)_{K_2}|^p}  \, \rmd x \right)^{1/p} + |(u)_{K_1} - (u)_{K_2}|\,.
\end{split}
\end{equation*}
Then using the Sobolev embedding theorem in Lemma \ref{lma:SobolevEmbeddingForDualPair} adapted for $U$ and $\nu$ and applied to cubes, we have
\begin{equation*}
\left( \fint_{K_h} {|u(x)-(u)_{K_1}|^p}  \, \rmd x \right)^{1/p} \leq \frac{C 2^{-k(s+\veps)}}{\veps^{1/\eta}} \left( \fint_{\Kh} U^{\eta} \, \rmd \nu \right)^{1/\eta}\,, \qquad h \in \{1 , 2\}\,, C = C(n,s,p)\,.
\end{equation*}
Next, using H\"older's inequality and repeatedly using \eqref{eq:OffDiagRevHolderIneq:Proof1} and \eqref{eq:OffDiagRevHolderIneq:Proof2},
\begin{equation*}
\begin{split}
|(u)_{K_1} - (u)_{K_2}| 
	&\leq \fint_{K_1} \fint_{K_2} |u(x)-u(y)| \, \rmd y \, \rmd x \\
	&\leq \left( \fint_{K_1} \fint_{K_2} |u(x)-u(y)|^{\eta} \, \rmd y \, \rmd x \right)^{1/\eta} \\
	&\leq C \left( \frac{1}{\dist(K_1,K_2)^{n-\veps p} \nu(\cK) } \int_{K_1} \int_{K_2} |u(x)-u(y)|^{\eta} \, \rmd y \, \rmd x \right)^{1/\eta} \\
	&\leq C \left( \fint_{\cK} |u(x)-u(y)|^{\eta} \, \rmd \nu \right)^{1/\eta} \leq C \dist(K_1,K_2)^{s+\veps} \left( \fint_{\cK} U^{\eta} \, \rmd \nu \right)^{1/\eta}\,,
\end{split}
\end{equation*}
where $C = C(\texttt{data})$. Combining the three above inequalities with \eqref{eq:OffDiagRevHolderIneq:Proof4} and using the inequality $(a+b+c)^r \leq 3^{r-1}(a^r + b^r + c^r)$ valid for $r \geq 1$ and $a$, $b$, $c \geq 0$ gives
\begin{equation*}
\begin{split}
\left( \fint_{\cK} H^{p'} \, \rmd \nu \right)^{1/p'} 
	&\leq C_{nd} \left( \fint_{\cK} U^{\eta} \, \rmd \nu \right)^{(p-1)/\eta} \\
	& + \frac{C_{nd}}{\veps^{(p-1)/\eta}} \left( \frac{2^{-k}}{\dist(K_1,K_2)} \right)^{(p-1)(s+\veps)} \left[ \left( \fint_{\Kone} U^{\eta} \, \rmd \nu \right)^{(p-1)/\eta} + \left( \fint_{\Ktwo} U^{\eta} \, \rmd \nu \right)^{(p-1)/\eta} \right]\,.
\end{split}
\end{equation*}
To see that \eqref{eq:OffDiagRevHolderIneq} follows, recall that $\gamma = \frac{\eta}{p-1}$ and that by definition $U^{\eta} \leq H^{\gamma}$ pointwise.
\end{proof}

Note that this lemma holds for all functions $u \in W^{s,p}(\bbR^n) \cap L^{\infty}(\bbR^n)$ and for every integer $k$, and in the case $a \equiv 0$ it holds for every $u \in W^{s,p}(\bbR^n)$. We now apply it in order to begin the level set estimate.

\begin{corollary}
Let $k \geq k_0$ be an integer, and suppose that $\cK \in \Delta_k$ satisfies $\dist(K_1,K_2) \geq 2^{-k}$. Assume that
\begin{equation*}
\left( \fint_{\cK} H^{p'} \, \rmd \nu \right)^{1/p'} \geq \lambda\,,
\end{equation*}
and that the constant $\kappa$ introduced in \eqref{eq:LevelSetProof:1} satisfies
\begin{equation}\label{eq:KappaCondition2}
\kappa \in (0,\kappa_1]\,, \qquad \kappa_1 := \frac{\veps^{1/\gamma}}{2^{1/\gamma} 3 C_{nd}}\,,
\end{equation}
where $C_{nd} \equiv C_{nd}(\texttt{data})$ has been defined in Lemma \ref{lma:OffDiagRevHolderIneq}. Then
\begin{equation}\label{eq:OffDiagLevelSetEstimate1}
\begin{split}
\nu(\cK) &\leq \frac{3^{\gamma} C_{nd}^{\gamma}}{ \lambda^{\gamma}} \intdm{\cK \cap \{ H > \kappa \lambda \}}{H^{\gamma}}{\nu} \\
	&\quad + \frac{3^{\gamma} C_{nd}^{\gamma}}{\veps \lambda^{\gamma}} \left( \frac{2^{-k}}{\dist(K_1,K_2)} \right)^{\eta(s+\veps)}  \left[ \frac{\nu(\cK)}{\nu(\Kone)} \intdm{\Kone \cap \{ H > \kappa \lambda \} }{H^{\gamma}}{\nu} + \frac{\nu(\cK)}{\nu(\Ktwo)} \intdm{\Ktwo \cap \{ H > \kappa \lambda \} }{H^{\gamma}}{\nu} \right]\,.
\end{split}
\end{equation}
In particular, \eqref{eq:OffDiagLevelSetEstimate1} holds whenever $\cK \in \Hndlambda$.
\end{corollary}

\begin{proof}
Applying the inequality $(a+b+c)^r \leq 3^{r-1} (a^r + b^r + c^r)$ to the conclusion \eqref{eq:OffDiagRevHolderIneq} in Lemma \ref{lma:OffDiagRevHolderIneq} and recalling that $(p-1) \gamma = \eta$, we get
\begin{equation}\label{eq:OffDiagLevelSetEstimate1:Proof1}
\begin{split}
\frac{\lambda^{\gamma}}{3^{\gamma-1} C_{nd}^{\gamma}} 
	\leq \fint_{\cK} H^{\gamma} \, \rmd \nu + \frac{1}{\veps} \left( \frac{2^{-k}}{\dist(K_1,K_2)} \right)^{\eta(s+\veps)} \left( \fint_{\Kone} H^{\gamma} \, \rmd \nu +  \fint_{\Ktwo} H^{\gamma} \, \rmd \nu  \right)\,.
\end{split}
\end{equation}
To estimate the integrals on the right-hand side, we use \eqref{eq:KappaCondition2} to get
\begin{equation*}
\fint_E H^{\gamma} \, \rmd \nu \leq \kappa_1^{\gamma} \lambda^{\gamma} + \frac{1}{\nu(E)} \intdm{E \cap \{ H > \kappa \lambda \}}{H^{\gamma}}{\nu}
\end{equation*}
for $E \in \{ \cK, \Kone, \Ktwo \}$, and therefore
\begin{equation*}
\begin{split}
\frac{\lambda^{\gamma}}{3^{\gamma-1}C_{nd}^{\gamma}}
 &\leq \frac{3 \kappa_1^{\gamma} \lambda^{\gamma}}{\veps} + \frac{1}{\nu(\cK)} \intdm{\cK \cap \{ H > \kappa \lambda \}}{H^{\gamma}}{\nu} \\
	&\quad + \frac{1}{\veps} \left( \frac{2^{-k}}{\dist(K_1,K_2)} \right)^{\eta(s+\veps)}  \left[ \frac{1}{\nu(\Kone)} \intdm{\Kone \cap \{ H > \kappa \lambda \} }{H^{\gamma}}{\nu} + \frac{1}{\nu(\Ktwo)} \intdm{\Ktwo \cap \{ H > \kappa \lambda \} }{H^{\gamma}}{\nu} \right]\,.
\end{split}
\end{equation*}
Now the estimate \eqref{eq:OffDiagLevelSetEstimate1} follows by using \eqref{eq:KappaCondition2} in the above estimate and absorbing terms.
\end{proof}

\subsubsection{Collections of Off-Diagonal Cubes}
We now split the collection $\Hndlambda$ into collections where the average of $H^{\gamma}$ is large and where the average of $H^{\gamma}$ is small. The splitting is chosen based on the first exit-time argument.

Consider
\begin{equation}\label{eq:DefnGoodPart}
\mathscr{G}_{\lambda}^h := \left\{ \cK \in \Hndlambda \, : \, \fint_{\Kh} H^{\gamma} \, \rmd \nu \leq(10n)^{n+p} \kappa^{\gamma} \lambda^{\gamma} \right\}
\end{equation}
and
\begin{equation}\label{eq:DefnBadPart}
\mathscr{B}_{\lambda}^h := \left\{  \cK \in \Hndlambda \, : \, \fint_{\Kh} H^{\gamma} \, \rmd \nu > (10n)^{n+p} \kappa^{\gamma} \lambda^{\gamma} \right\}
\end{equation}
for $h \in \{ 1, 2 \}$ and $\kappa$ and $\gamma$ introduced in  \eqref{eq:LevelSetProof:1} and Corollary \ref{cor:RevHolderFinal} respectively. We further define 
\begin{equation*}
\GoodLambda := \mathscr{G}_{\lambda}^1 \cap \mathscr{G}_{\lambda}^2\,, \qquad \text{ and } \qquad \mathscr{B}_{\lambda} := \mathscr{B}_{\lambda}^1 \cup \mathscr{B}_{\lambda}^2\,.
\end{equation*}
We further split the set $\mathscr{B}_{\lambda}$ in order to remove cubes that are already covered by  the diagonal balls in \eqref{eq:LevelSetProof:4}-\eqref{eq:LevelSetProof:5}:
\begin{equation}\label{eq:BadSetDecomp1}
\BadLambdaD := \left\{ \cK \in \mathscr{B}_{\lambda} \, : \, \cK \subset \bigcup_j 10 \cB_j \right\}\,, \qquad \BadLambdaND := \mathscr{B}_{\lambda} \setminus \BadLambdaD\,.
\end{equation}
Thus we have the decomposition into disjoint families
\begin{equation*}
\Hndlambda = \GoodLambda \cup \BadLambdaD \cap \BadLambdaND\,.
\end{equation*}

It turns out that the family $\GoodLambda$ is ``good" in the sense that the measures of cubes in this collection are estimated by the $\mu$-measure of the level set $\{ H > \kappa \lambda \}$, where $\rmd \mu := H^{p'} \rmd \nu$. The cubes belonging to $\BadLambdaND$ are ``bad" because there is no available control of the size of $H^{p'}$ on diagonal cubes via the exit-time argument. They will instead be dealt with using combinatorial arguments, and the cutoff $(10n)^{n+p} \kappa^{\gamma} \lambda^{\gamma}$ is chosen for precisely this purpose.

\begin{lemma}[First (easier) off-diagonal estimate]\label{lma:EasyOffDiagSum}
We have
\begin{equation*}
\sum_{\cK \in \GoodLambda} \nu(\cK) \leq \frac{6^{\gamma} C_{nd}^{\gamma}}{\lambda^{\gamma}} \intdm{\cB(x_0,\alpha) \cap \{ H > \kappa \lambda\}}{H^{\gamma}}{\nu}
\end{equation*}
whenever the number $\kappa \in (0,1]$ satisfies
\begin{equation}\label{eq:KappaCondition3}
\kappa \in (0,\kappa_2]\,, \qquad \kappa_2 := \frac{\veps^{1/\gamma}}{8^{1/\gamma} 3 C_{nd} (10n)^{(n+p)/\gamma}}\,.
\end{equation}
The constant $C_{nd}$ has been defined in Lemma \ref{lma:OffDiagRevHolderIneq}.
\end{lemma}

\begin{proof}
Since the cubes $\cK \in \GoodLambda$ are disjoint and since \eqref{eq:LevelSetProof:CubesAreCovering} holds, it suffices to show that
\begin{equation}\label{eq:GoodLambda:Proof1}
\nu(\cK) \leq \frac{6^{\gamma} C_{nd}^{\gamma}}{\lambda^{\gamma}} \intdm{\cK \cap \{ H > \kappa \lambda\}}{H^{\gamma}}{\nu}
\end{equation}
for every $\cK \in \GoodLambda$. Since $\cK \in \GoodLambda$, using \eqref{eq:KappaCondition3} we have
\begin{equation}\label{eq:GoodLambda:Proof2}
\begin{split}
\frac{3^{\gamma} C_{nd}^{\gamma}}{\veps \lambda^{\gamma}} \left( \frac{2^{-k}}{\dist(K_1,K_2)} \right)^{\eta(s+\veps)} & \frac{\nu(\cK)}{\nu(\Kh)} \intdm{\Kh \cap \{ H > \kappa \lambda \}}{H^{\gamma}}{\nu} \\
	&\leq \frac{3^{\gamma} C_{nd}^{\gamma}}{\veps \lambda^{\gamma}} \nu(\cK) \fint_{\Kh} H^{\gamma} \, \rmd \nu \leq \nu(\cK) \frac{3^{\gamma} C_{nd}^{\gamma}}{\veps \lambda^{\gamma}} (10n)^{n+p} \kappa^{\gamma} \lambda^{\gamma} \leq \frac{\nu(\cK)}{8}\,.
\end{split}
\end{equation}
Using this estimate for $h \in \{1, 2\}$ in \eqref{eq:OffDiagLevelSetEstimate1} and then absorbing terms gives \eqref{eq:GoodLambda:Proof1}.
\end{proof}

\subsubsection{Determining $\kappa$}
At this point the constant $\kappa$ introduced in \eqref{eq:LevelSetProof:1} can be completely determined. We choose 
\begin{equation}\label{eq:KappaChoice}
\kappa := \min\{ \kappa_0, \kappa_1, \kappa_2 \} \equiv \min \left\{ \frac{\veps^{1/p'}}{(2C_d)^{1/p}}\,, \frac{\veps^{1/\gamma}}{2^{1/\gamma} 3 C_{nd}}\,, \frac{\veps^{1/\gamma}}{8^{1/\gamma} 3 C_{nd} (10n)^{(n+p)/\gamma}} \right\}\,,
\end{equation}
so that conditions \eqref{eq:KappaCondition1},  \eqref{eq:KappaCondition2}, and  \eqref{eq:KappaCondition3} are satisfied. Since the constant $C_d$ defined in Theorem \ref{thm:Measure} depends only on $n$ and $p$, and since $C_{nd}$ defined in Lemma \ref{lma:OffDiagRevHolderIneq} depends only on $\texttt{data}$, we conclude that there exists $C_{\kappa} = C_{\kappa}(\texttt{data})$ such that
\begin{equation}\label{eq:KappaConditionFinal}
\kappa \geq \frac{\veps^{1/\gamma}}{C_{\kappa}}\,.
\end{equation}

\subsubsection{Summation in $\BadLambdaND$}
Dealing with the cubes from $\BadLambdaND$ requires delicate estimates and combinatorial arguments. We will first set up notation designed to describe the cubes more precisely. We will then demonstrate that the choice of cutoff $(10n)^{n+p} \kappa^{\gamma} \lambda^{\gamma}$ gives us an upper bound on the distance of a cube from the diagonal.

Define the ``problematic" projections of the ``bad" off-diagonal set
\begin{equation*}
\pi_h \mathscr{B}_{\lambda} := \{ \Kh \, : \, \cK \in \mathscr{B}_{\lambda}^h \}\,, \quad h \in \{ 1,2 \}\,.
\end{equation*}
Since all cubes belonging to the collection $\pi_1 \mathscr{B}_{\lambda} \cup \pi_2 \mathscr{B}_{\lambda}$ are dyadic cubes, it follows that a disjoint subfamily of cubes always exists. We denote this disjoint subfamily of $\pi_1 \mathscr{B}_{\lambda} \cup \pi_2 \mathscr{B}_{\lambda}$ by $\pi \mathscr{B}_{\lambda}$. By definition, all cubes of $\pi \mathscr{B}_{\lambda}$ belong to $\pi_1 \mathscr{B}_{\lambda} \cup \pi_2 \mathscr{B}_{\lambda}$ and are therefore diagonal cubes.

\begin{lemma}\label{lma:CombinatorialLemma}
Let $\cK \in \BadLambdaND$ be a cube such that $\Kh \subset \cM$ for some $\cM \in \pi \mathscr{B}_{\lambda}$ and some $h \in \{1,2\}$. Then $\dist(K_1,K_2) \geq 2^{-k(\cM)}$.
\end{lemma}

\begin{proof}
First, consider $\cM \in \pi \mathscr{B}_{\lambda}$. Take the diagonal ball $\cB(\cM) \equiv \cB(x_{\cM}, 2^{-(k(\cM)+1)})$ with $x_{\cM}$ being the center of $\cM$. Thus,
\begin{equation*}
\cB(\cM) \subset \cM \subset \sqrt{n} \cB(\cM)\,.
\end{equation*}
Thus, by \eqref{eq:MeasureDoublingProperty}, H\"older's inequality, and by the definition of $\pi \mathscr{B}_{\lambda}$
\begin{equation*}
\begin{split}
(10n)^{(n+p)/\gamma} \kappa \lambda &< \left( \fint_{\cM} H^{\gamma} \, \rmd \nu \right)^{1/\gamma} \\
	&\leq \left( \frac{\nu(10n \cB(\cM))}{\nu(\cB(\cM))}\fint_{\cM} H^{\gamma} \, \rmd \nu \right)^{1/\gamma} \\
	&\leq (10n)^{(n+p)\gamma}\left( \fint_{\cM} H^{p'} \, \rmd \nu \right)^{1/p'}\,.
\end{split}
\end{equation*}
By the definition of $D_{\kappa \lambda}$ in \eqref{eq:LevelSetProof:2} we have $(x_{\cM},x_{\cM}) \in D_{\kappa \lambda}$, and then the exit-time condition \eqref{eq:LevelSetProof:4} gives $\cB(\cM) \subset \cB(x_{\cM}, R(\cM))$. We can use the exit time condition since the radius of the ball $10 n \cB(\cM)$ is smaller than $\frac{\alpha-\beta}{40^n}$, which is in turn a consequence of the fact that $k(\cM) + 1 \geq k_0$ and $k_0$ has been chosen as in \eqref{eq:MinSideLength}. Therefore, by \eqref{eq:LevelSetProof:5}
\begin{equation}\label{eq:OffDiagBadBallInc}
10n \cB(\cM) \subset \bigcup_j 10 \cB_j\,.
\end{equation}
Now, assume by contradiction that $\dist(K_1, K_2) < 2^{-k(\cM)}$. We will show that
\begin{equation}\label{eq:OffDiagBadBallInc2}
\cK \subset 10 n \cB(\cM)\,,
\end{equation}
which then contradicts the assumption $\cM \in \BadLambdaND$ by \eqref{eq:OffDiagBadBallInc}. To show \eqref{eq:OffDiagBadBallInc2} we use Proposition \ref{prop:Cubes} and that $\Kh \subset \cM$ to get
\begin{equation*}
\dist(\cK,\cM) \leq \dist(\cK,\Kh) = \dist(K_1, K_2) \leq 2^{-k(\cM)}\,.
\end{equation*}
Again using Proposition \ref{prop:Cubes} we have $k(\cK) = k(\Kh)$ and $k(\cK) \geq k(\cM)$. Therefore, since $\cM \subset \sqrt{n} \cB(\cM)$ and the radius of $\cB(\cM)$ is $2^{-(k(\cM)+1)}$, then \eqref{eq:OffDiagBadBallInc2} must hold.
\end{proof}

\begin{lemma}[Second (harder) off-diagonal estimate]\label{lma:HardOffDiagSum} 
There exists a constant $C = C(\texttt{data})$ such that the estimate
\begin{equation}\label{eq:HardOffDiagSum}
\sum_{\cK \in \BadLambdaND} \nu(\cK) \leq \frac{C}{\lambda^{\gamma}} \intdm{\cB(x_0,\alpha) \cap \{ H > \kappa \lambda \} }{H^{\gamma}}{\nu}
\end{equation}
holds, where $\kappa$ is as in \eqref{eq:KappaChoice}.
\end{lemma}

\begin{proof}
\underline{Step 1: Classification}. We classify cubes from $\BadLambdaND$ according to the location of their projections, their size, and their distance from the diagonal. Lemma \ref{lma:CombinatorialLemma}, in summary, correlates the distance of a cube from the diagonal with the location of its projection. This allows us to consider only cubes in $\pi \mathscr{B}_{\lambda}$.

We will partition $\BadLambdaND$ into suitable disjoint subfamilies. Define the collections
\begin{equation*}
\mathscr{B}_{\lambda, nd}^h := \BadLambdaND \cap \mathscr{B}_{\lambda}^h\,, \qquad h \in \{ 1, 2 \}
\end{equation*}
For every $\cM \in \pi \mathscr{B}_{\lambda}$, set
\begin{equation*}
\mathscr{B}_{\lambda, nd}^h(\cM) := \{ \cK \in \BadLambdaND \, : \, \Kh \subset \cM \}\,, \quad h \in \{ 1,2 \}\,.
\end{equation*}
Thus, we have the decomposition into disjoint subcollections
\begin{equation}\label{eq:HardOffDiagEst:Step1Decomp1}
\mathscr{B}_{\lambda, nd}^h = \bigcup_{\cM \in \pi \mathscr{B}_{\lambda}} \mathscr{B}_{\lambda, nd}^h(\cM)\,.
\end{equation}
These subcollections are disjoint in the sense that for $\cM_1$, $\cM_2 \in \pi \mathscr{B}_{\lambda}$, $\mathscr{B}_{\lambda, nd}^h(\cM_1) \cap \mathscr{B}_{\lambda, nd}^h(\cM_2) \neq \emptyset$ implies $\cM_1 = \cM_2$. This follows from the definition of $\mathscr{B}_{\lambda, nd}^h(\cM)$ and from the fact that all elements in $\pi \mathscr{B}_{\lambda}$ are pairwise disjoint.

Next we classify cubes according to their size. For each $\cK \in \mathscr{B}_{\lambda, nd}^h(\cM)$ we have $k(\cK) = k(\Kh) \geq k(\cM)$, so we can define the collections
\begin{equation*}
[\mathscr{B}_{\lambda,nd}^h(\cM)]_i := \{ \cK \in \mathscr{B}_{\lambda,nd}^h(\cM) \, : \, k(\cK) = i + k(\cM) \}\,, \qquad h \in \{1, 2\}\,, \quad i \in \bbZ_+\,.
\end{equation*}
We again have a decomposition in mutually disjoint subcollections
\begin{equation*}
\mathscr{B}_{\lambda,nd}^h(\cM) = \bigcup_{i \in \bbZ_+} [\mathscr{B}_{\lambda,nd}^h(\cM)]_i\,,
\end{equation*}
in the sense that $[\mathscr{B}_{\lambda,nd}^h(\cM)]_{i_1} \cap [\mathscr{B}_{\lambda,nd}^h(\cM)]_{i_2} \neq \emptyset$ implies that $i_1 = i_2$.

We further classify the cubes according to their distance from the diagonal. We now use the combinatorial information obtained from Lemma \ref{lma:CombinatorialLemma}. Take $\cM \in \pi \mathscr{B}_{\lambda}$. If $\cK \in \mathscr{B}_{\lambda,nd}^h(\cM)$; that is, if $\Kh \subset \cM$, then by Lemma \ref{lma:CombinatorialLemma} it follows that $\dist(K_1,K_2) \geq 2^{-k(\cM)}$. This leads us to the definition of the subcollections
\begin{equation*}
[\mathscr{B}_{\lambda,nd}^h(\cM)]_{i,j} := \{ \cK \in [\mathscr{B}_{\lambda,nd}^h(\cM)]_i \, : \, 2^{j-k(\cM)} \leq \dist(K_1,K_2) < 2^{j+1-k(\cM)} \}\,, \qquad h \in \{1,2\}\,, \quad i,j \in \bbZ_+\,,
\end{equation*}
We yet again have the decomposition in mutually disjoint subcollections
\begin{equation*}
\mathscr{B}_{\lambda,nd}^h(\cM) = \bigcup_{i,j \in \bbZ_+} [\mathscr{B}_{\lambda,nd}^h(\cM)]_{i,j}
\end{equation*}
and these are mutually disjoint in the sense that $[\mathscr{B}_{\lambda,nd}^h(\cM)]_{i_1, j_1} \cap [\mathscr{B}_{\lambda,nd}^h(\cM)]_{i_2, j_2} \neq \emptyset$ implies that $(i_1,j_1) = (i_2,j_2)$. In summary, we have the decomposition 
\begin{equation*}
\mathscr{B}_{\lambda,nd}^h = \bigcup_{\cM \in \pi \mathscr{B}_{\lambda}} \bigcup_{i,j \in \bbZ_+} [\mathscr{B}_{\lambda,nd}^h(\cM)]_{i,j}\,.
\end{equation*}
\underline{Step 2: Summation and Further Partitioning}. Fix $\cM \in \pi \mathscr{B}_{\lambda}$. Our goal in this step is to show that for $h \in \{ 1, 2 \}$
\begin{equation}\label{eq:HardOffDiagSum:Step2Estimate}
\frac{1}{\veps} \sum_{\cK \in \mathscr{B}_{\lambda,nd}^h(\cM)} \frac{\nu(\cK)}{\nu(\Kh)} \left( \frac{2^{-k}}{\dist(K_1,K_2)}\right)^{\eta(s+\veps)} \intdm{\Kh \cap \{ H > \kappa \lambda \}}{H^{\gamma}}{\nu} \leq \frac{C(\texttt{data})}{s^2} \intdm{\cM \cap \{ H > \kappa \lambda \}}{H^{\gamma}}{\nu}\,.
\end{equation}
To begin, note that for $\cK \in \mathscr{B}_{\lambda,nd}^h$ we have $\dist(K_1,K_2) \geq 2^{-k(\cK)}$, and so we can apply the inequality from Lemma \ref{lma:TwoDimensionalConstants} to get that 
\begin{equation*}
\frac{1}{\veps} \frac{\nu(\cK)}{\nu(\Kh)} \leq C_{dd} \left( \frac{2^{-k}}{\dist(K_1,K_2)}\right)^{n-\veps p}
\end{equation*}
for $h \in \{ 1, 2 \}$, and further if $\cK \in [\mathscr{B}_{\lambda, nd}^h]_{i,j}$ then
\begin{equation*}
\frac{2^{-k(\cK)}}{\dist(K_1,K_2)} = \frac{1}{2^i}\frac{2^{-k(\cM)}}{\dist(K_1,K_2)} \leq \frac{1}{2^{i+j}}\,.
\end{equation*}
With these we can estimate
\begin{equation}\label{eq:HardOffDiagSum:Step2Estimate:1}
\begin{split}
\frac{1}{\veps} &\sum_{\cK \in \mathscr{B}_{\lambda,nd}^h(\cM)} \frac{\nu(\cK)}{\nu(\Kh)} \left( \frac{2^{-k}}{\dist(K_1,K_2)}\right)^{\eta(s+\veps)} \intdm{\Kh \cap \{ H > \kappa \lambda \}}{H^{\gamma}}{\nu} \\
	&\leq C_{dd} \sum_{\cK \in \mathscr{B}_{\lambda,nd}^h(\cM)} \left( \frac{2^{-k}}{\dist(K_1,K_2)}\right)^{n+\eta(s+\veps)-\veps p} \intdm{\Kh \cap \{ H > \kappa \lambda \}}{H^{\gamma}}{\nu} \\
	&= C_{dd} \sum_{i,j = 0}^{\infty} \sum_{\cK \in [\mathscr{B}_{\lambda,nd}^h(\cM)]_{i,j}} \left( \frac{2^{-k}}{\dist(K_1,K_2)}\right)^{n+\eta(s+\veps)-\veps p} \intdm{\Kh \cap \{ H > \kappa \lambda \}}{H^{\gamma}}{\nu} \\
	&\leq C_{dd} \sum_{i,j = 0}^{\infty} \left( \frac{1}{2^{i+j}} \right)^{n+\eta(s+\veps)-\veps p}  \sum_{\cK \in [\mathscr{B}_{\lambda,nd}^h(\cM)]_{i,j}} \intdm{\Kh \cap \{ H > \kappa \lambda \}}{H^{\gamma}}{\nu}\,.
\end{split}
\end{equation}
In order to evaluate the last sum we must further partition $[ \mathscr{B}_{\lambda,nd}^h(\cM)]_{i,j} $. In order to do so, note that for each $i \in \bbZ_+$ the cube $\cM$ contains precisely $2^{2ni}$ disjoint dyadic cubes belonging to $\Delta_{i+k(\cM)}$ and precisely $2^{ni}$ disjoint dyadic diagonal cubes from $\Delta^d_{i+k(\cM)}$; see the definition \eqref{eq:DefnOfCubeCollection} and the text immediately following. Thus, $\cM$ contains at most $2^{ni}$ dyadic cubes from the class $\Delta^d_{i+k(\cM)} \cap (\pi_1 \mathscr{B}_{\lambda} \cap \pi_2 \mathscr{B}_{\lambda})$. In any case we consider all the diagonal cubes from $\cM$ in  $\Delta^d_{i+k(\cM)}$ and label them as
\begin{equation*}
\{ \widetilde{\cM} \in \Delta^d_{i+k(\cM)} \, : \, \widetilde{\cM} \subset \cM \} \equiv \{ \cM_i^m \, : \, 1 \leq m \leq 2^{ni} \}\,,
\end{equation*}
so that
\begin{equation}\label{eq:HardOffDiagSum:Key1}
\sum_{m=1}^{2^{ni}} \intdm{\cM_i^m \cap \{ H > \kappa \lambda \}}{H^{\gamma}}{\nu} \leq \intdm{\cM \cap \{ H > \kappa \lambda \}}{H^{\gamma}}{\nu}\,.
\end{equation}
For any $\cK \in [ \mathscr{B}_{\lambda,nd}^h(\cM)]_{i,j}$ with $h \in \{1,2\}$ there exists a unique diagonal cube from the collection $\Delta_k^d$, which we denote by $\cM_i^m(\cK)$, such that $\Kh = \cM_i^m(\cK)$. Then we can split $[ \mathscr{B}_{\lambda,nd}^h(\cM)]_{i,j}$ into subsets
\begin{equation*}
[ \mathscr{B}_{\lambda,nd}^h(\cM)]_{i,j,m} := \{ \cK \in [ \mathscr{B}_{\lambda,nd}^h(\cM)]_{i,j} \, : \, \Kh = \cM_i^m \}\,, \qquad m \in \{ 1, \ldots, 2^{ni} \}\,.
\end{equation*}
Since $\mathscr{B}_{\lambda, nd}^1$ is a family of dyadic cubes, if $\cK_1, \cK_2 \in [ \mathscr{B}_{\lambda,nd}^h(\cM)]_{i,j,m}$ and $\cK_1 \neq \cK_2$, then $\Ktwo_1 \cap \Ktwo_2 = \emptyset$, else the two cubes would coincide. 

Thus for each $i$, $j \geq 0$ and for $m \in \{1 , \ldots , 2^{ni} \}$ the number of cubes $\cK$ of side length $2^{i + k(\cM)}$ with $\dist(K_1, K_2) \in [2^{j-k(\cM)}, 2^{j+1-k(\cM)})$ that project into $\cM_i^m$ is bounded from above by $C(n) \left( \frac{\text{length of half-open interval defining the distance from the diagonal} }{\text{cube side length}} \right)^n$; that is,
\begin{equation}\label{eq:HardOffDiagSum:Key2}
\# [ \mathscr{B}_{\lambda,nd}^h(\cM)]_{i,j,m} \leq C(n) 2^{n(i+j)}\,, \qquad h \in \{1, 2\}\,.
\end{equation}
Thus, using \eqref{eq:HardOffDiagSum:Key1} and \eqref{eq:HardOffDiagSum:Key2} we estimate
\begin{equation*}
\begin{split}
\sum_{\cK \in  [\mathscr{B}_{\lambda,nd}^h(\cM)]_{i,j} } \int_{\Kh \cap \{H > \kappa \lambda \}} H^{\gamma} \, \rmd \nu &= \sum_{m=1}^{2^{ni}} \sum_{\cK \in  [\mathscr{B}_{\lambda,nd}^h(\cM)]_{i,j,m} } \int_{\cM^m_i \cap \{H > \kappa \lambda \}} H^{\gamma} \, \rmd \nu \\
	&\leq C(n) 2^{n(i+j)} \sum_{m=1}^{2^{ni}} \int_{\cM^m_i \cap \{H > \kappa \lambda \}} H^{\gamma} \, \rmd \nu \\
	&\leq C(n) 2^{n(i+j)} \int_{\cM \cap \{H > \kappa \lambda \}} H^{\gamma} \, \rmd \nu\,.
\end{split}
\end{equation*}
Then using the series estimate \eqref{eq:GeometricSeriesEstimate} it follows that 
\begin{equation*}
\begin{split}
\sum_{i,j=0}^{\infty} & \left( \frac{1}{2^{i+j}} \right)^{n+ \eta(s+\veps) - \veps p} \sum_{\cK \in [\mathscr{B}_{\lambda,nd}^h(\cM) ]_{i,j}} \int_{\Kh \cap \{ H > \kappa \lambda \} } H^{\gamma} \, \rmd \nu \\
	&\leq C(n) \sum_{i,j =0}^{\infty} \left( \frac{1}{2^{i+j}} \right)^{\eta(s+\veps) - \veps p} \int_{\cM \cap \{H > \kappa \lambda \}} H^{\gamma} \, \rmd \nu \\
	&\leq C(n) \left[ \frac{2^{\eta(s+\veps)-\veps p}}{\ln(2) (\eta(s+\veps)-\veps p) } \right]^2 \int_{\cM \cap \{H > \kappa \lambda \}} H^{\gamma} \, \rmd \nu \\
	&\leq  \frac{C(n,p) }{[ \eta(s+\veps)-\veps p]^2} \int_{\cM \cap \{H > \kappa \lambda \}} H^{\gamma} \, \rmd \nu
	\leq \frac{C(n,p)}{s^2} \int_{\cM \cap \{H > \kappa \lambda \}} H^{\gamma} \, \rmd \nu\,.
\end{split}
\end{equation*}
In the last line we used that $\eta > 1$ and that $\veps < \frac{s}{p}$, which gives $\eta(s+\veps)-\veps p > \frac{s}{p}$. Combining this inequality with \eqref{eq:HardOffDiagSum:Step2Estimate:1} results in \eqref{eq:HardOffDiagSum:Step2Estimate}.

\underline{Step 3: Further Summation}. Let $\cK \in \mathscr{B}_{\lambda, nd}^1$. Then, either $\cK \in \mathscr{G}_{\lambda}^2$ or $\cK \in \mathscr{B}_{\lambda}^2$ (see \eqref{eq:DefnGoodPart} and \eqref{eq:DefnBadPart}). If $\cK \in \mathscr{G}_{\lambda}^2$, then using the almost-reverse H\"older inequality \eqref{eq:OffDiagLevelSetEstimate1} along with the estimate \eqref{eq:GoodLambda:Proof2} for cubes in $\mathscr{G}_{\lambda}^2$ and reabsorbing terms,
\begin{equation*}
\nu(\cK) \leq \frac{6^{\gamma} C_{nd}^{\gamma}}{ \lambda^{\gamma}} \intdm{\cK \cap \{ H > \kappa \lambda \}}{H^{\gamma}}{\nu} 
	+ \frac{6^{\gamma} C_{nd}^{\gamma}}{\veps \lambda^{\gamma}} \left( \frac{2^{-k}}{\dist(K_1,K_2)} \right)^{\eta(s+\veps)}  \frac{\nu(\cK)}{\nu(\Kone)} \intdm{\Kone \cap \{ H > \kappa \lambda \} }{H^{\gamma}}{\nu}\,.
\end{equation*}
On the other hand, if $\cK \in \mathscr{B}_{\lambda}^2$ then we can use the almost-reverse H\"older inequality\eqref{eq:OffDiagLevelSetEstimate1} to obtain
\begin{equation*}
\begin{split}
\nu(\cK) &\leq \frac{3^{\gamma} C_{nd}^{\gamma}}{ \lambda^{\gamma}} \intdm{\cK \cap \{ H > \kappa \lambda \}}{H^{\gamma}}{\nu} \\
	&\quad + \frac{3^{\gamma} C_{nd}^{\gamma}}{\veps \lambda^{\gamma}} \left( \frac{2^{-k}}{\dist(K_1,K_2)} \right)^{\eta(s+\veps)}  \left[ \frac{\nu(\cK)}{\nu(\Kone)} \intdm{\Kone \cap \{ H > \kappa \lambda \} }{H^{\gamma}}{\nu} + \frac{\nu(\cK)}{\nu(\Ktwo)} \intdm{\Ktwo \cap \{ H > \kappa \lambda \} }{H^{\gamma}}{\nu} \right]\,.
\end{split}
\end{equation*}
A similar reasoning holds for $\cK \in \mathscr{B}_{\lambda}^2$. Summing over the cubes $\cK \in \BadLambdaND = \mathscr{B}_{\lambda,nd}^1 \cup\mathscr{B}_{\lambda,nd}^2$ gives us
\begin{equation}\label{eq:HardOffDiagSum:Step3Est}
\begin{split}
\sum_{\cK \in \BadLambdaND} &\nu(\cK) \leq \frac{6^{\gamma} C_{nd}^{\gamma}}{ \lambda^{\gamma}} \sum_{\cK \in \BadLambdaND} \intdm{\cK \cap \{ H > \kappa \lambda \}}{H^{\gamma}}{\nu} \\
	&+ \frac{6^{\gamma} C_{nd}^{\gamma}}{\veps \lambda^{\gamma}} \sum_{\cK \in \mathscr{B}_{\lambda,nd}^1} \left( \frac{2^{-k}}{\dist(K_1,K_2)} \right)^{\eta(s+\veps)}  \frac{\nu(\cK)}{\nu(\Kone)} \intdm{\Kone \cap \{ H > \kappa \lambda \} }{H^{\gamma}}{\nu} \\
	&+ \frac{6^{\gamma} C_{nd}^{\gamma}}{\veps \lambda^{\gamma}} \sum_{\cK \in \mathscr{B}_{\lambda,nd}^2} \left( \frac{2^{-k}}{\dist(K_1,K_2)} \right)^{\eta(s+\veps)}  \frac{\nu(\cK)}{\nu(\Ktwo)} \intdm{\Ktwo \cap \{ H > \kappa \lambda \} }{H^{\gamma}}{\nu} \,.
\end{split}
\end{equation}
The point of this argument beginning in Step 3 is that terms involving the projections $\Kh$ appear if and only if $\cK \in \mathscr{B}_{\lambda,nd}^h$ for $h \in \{1,2\}$. We will now argue that the last two terms in the above inequality coincide.

For a cube $\cK = K_1 \times K_2 \in \Delta$, define
\begin{equation*}
\text{Symm}(\cK) := K_2 \times K_1\,.
\end{equation*}
Then by definition
\begin{equation}\label{eq:HardOffDiagSum:Symm1}
\pi_1(\cK) = \pi_1(K_1 \times K_2) = \pi_2(K_2 \times K_1) = \pi_2(\text{Symm}(\cK))\,.
\end{equation}
Further, by symmetry of $H$ we have
\begin{equation*}
\intdm{\cK}{H^{r}}{\nu} = \intdm{\text{Symm}(\cK)}{H^{r}}{\nu}\,, \qquad r \in [1,\infty)
\end{equation*}
and
\begin{equation}\label{eq:HardOffDiagSum:Symm2}
\intdm{\cK \cap \{ H > \kappa \lambda \} }{H^{r}}{\nu} = \intdm{\text{Symm}(\cK) \cap \{ H > \kappa \lambda \} }{H^{r}}{\nu}\,, \qquad r \in [1,\infty)\,.
\end{equation}
Therefore, by definition of $\mathscr{B}_{\lambda}^h$
\begin{equation*}
\cK \in \mathscr{B}_{\lambda}^1 \quad \Leftrightarrow \quad \text{Symm}(\cK) \in \mathscr{B}_{\lambda}^2\,,
\end{equation*}
and vice versa, and so
\begin{equation}\label{eq:HardOffDiagSum:Symm3}
\cK \in \mathscr{B}_{\lambda,nd}^1 \quad \Leftrightarrow \quad \text{Symm}(\cK) \in \mathscr{B}_{\lambda, nd}^2\,,
\end{equation}
and vice versa.
It follows from \eqref{eq:HardOffDiagSum:Symm1}-\eqref{eq:HardOffDiagSum:Symm3} that the last two terms in \eqref{eq:HardOffDiagSum:Step3Est} coincide. Recalling the decomposition \eqref{eq:HardOffDiagEst:Step1Decomp1}, \eqref{eq:HardOffDiagSum:Step3Est} can be written as
\begin{equation*}
\begin{split}
\sum_{\cK \in \BadLambdaND} &\nu(\cK) \leq \frac{C}{ \lambda^{\gamma}} \sum_{\cK \in \BadLambdaND} \intdm{\cK \cap \{ H > \kappa \lambda \}}{H^{\gamma}}{\nu} \\
	&+ \frac{C}{\veps \lambda^{\gamma}} \sum_{\cM \in \pi \mathscr{B}_{\lambda} } \sum_{\cK \in \mathscr{B}_{\lambda,nd}^1 (\cM)} \left( \frac{2^{-k}}{\dist(K_1,K_2)} \right)^{\eta(s+\veps)}  \frac{\nu(\cK)}{\nu(\Kone)} \intdm{\Kone \cap \{ H > \kappa \lambda \} }{H^{\gamma}}{\nu}\,,
\end{split}
\end{equation*}
for a constant $C \equiv C(\texttt{data})$. Using the estimate \eqref{eq:HardOffDiagSum:Step2Estimate} proved in Step 2 yields
\begin{equation*}
\begin{split}
\sum_{\cK \in \BadLambdaND} &\nu(\cK) \leq \frac{C}{ \lambda^{\gamma}} \sum_{\cK \in \BadLambdaND} \intdm{\cK \cap \{ H > \kappa \lambda \}}{H^{\gamma}}{\nu} + \frac{C}{\lambda^{\gamma}} \sum_{\cM \in \pi \mathscr{B}_{\lambda} } \intdm{\cM \cap \{ H > \kappa \lambda \} }{H^{\gamma}}{\nu}\,.
\end{split}
\end{equation*}
Then since the collections $\BadLambdaND$ and $\pi \mathscr{B}_{\lambda}$ are comprised of mutually disjoint dyadic cubes all of which are contained in $\cB(x_0,\alpha)$ (see \eqref{eq:LevelSetProof:CubesAreCovering}) we can estimate
\begin{equation*}
\sum_{\cK \in \BadLambdaND} \intdm{\cK \cap \{ H > \kappa \lambda \}}{H^{\gamma}}{\nu} + \sum_{\cM \in \pi \mathscr{B}_{\lambda} } \intdm{\cM \cap \{ H > \kappa \lambda \} }{H^{\gamma}}{\nu} \leq 2 \intdm{\cB(x_0,\alpha) \cap \{ H > \kappa \lambda \} }{H^{\gamma}}{\nu}\,.
\end{equation*}
Thus \eqref{eq:HardOffDiagSum} is proved, and the proof of Lemma \ref{lma:HardOffDiagSum} is complete.
\end{proof}

\subsubsection{Off-Diagonal Conclusion} The next lemma summarizes the off-diagonal estimate obtained for $H$.

\begin{lemma}\label{lma:OffDiagonalConclusion}
The inequality
\begin{equation}\label{eq:LevelSetProof:OffDiagConclusion}
\intdm{\cB(x_0,\beta) \cap \{ H > \lambda \} }{H^{p'}}{\nu} \leq 10^{n+p} \kappa^{p'} \lambda^{p'} \sum_j \nu(\cB_j) + C \lambda^{p'-\gamma} \intdm{\cB(x_0,\alpha) \cap \{ H > \kappa \lambda \} }{H^{\gamma}}{\nu}
\end{equation}
holds for a constant $C \equiv C(\texttt{data})$, while $\kappa$ has been chosen in \eqref{eq:KappaChoice} and satisfies \eqref{eq:KappaConditionFinal}.
\end{lemma}

\begin{proof}
We have the decomposition $\Hlambda = \Hdlambda \cup \Hndlambda$ with $\Hndlambda = \GoodLambda \cup \BadLambdaD \cup \BadLambdaND$. Recall from the relation \eqref{eq:DiagCubesContained} and from the definition \eqref{eq:BadSetDecomp1} that
\begin{equation*}
\bigcup_{\cK \in \Hdlambda} \cK \cup \bigcup_{\cK \in \BadLambdaD} \cK \subset \bigcup_{j} 10 \cB_j\,.
\end{equation*}
Thus
\begin{equation*}
\bigcup_{\cK \in \Hlambda} \subset \bigcup_{j} 10 \cB_j \cup \bigcup_{\cK \in \GoodLambda} \cK \cup \bigcup_{\cK \in \BadLambdaND} \cK \,.
\end{equation*}
Now, with this and with \eqref{eq:BoundOnHOutsideOfHLambda}, we estimate
\begin{equation*}
\intdm{\cB(x_0,\beta) \cap \{ H > \lambda \}}{H^{p'}}{\nu} \leq \sum_j \intdm{10 \cB_j \cap \{ H > \lambda \}}{H^{p'}}{\nu} + \sum_{\cK \in \GoodLambda \cup \BadLambdaND} \intdm{\cK \cap \{ H > \lambda \}}{H^{p'}}{\nu}\,.
\end{equation*}
Now, by the choice of Calder\'on-Zygmund splitting in \eqref{eq:CZSplittingForH}, for $\cK \in \GoodLambda \cup \BadLambdaND \subset \Hndlambda$
\begin{equation*}
\fint_{\cK} H^{p'} \, \rmd \nu \leq \frac{\nu(\widetilde{\cK})}{\nu(\cK)} \fint_{\widetilde{\cK}} H^{p'} \, \rmd \nu \leq \frac{\nu(\widetilde{\cK})}{\nu(\cK)} \lambda^{p'}\,.
\end{equation*}
Since $\cK \in \Hndlambda$ it follows from the definition \eqref{eq:DefnOfHndlambda} that $\dist(\widetilde{K}_1, \widetilde{K}_2) \geq 2^{-k(\cK)}$, so we can use the geometric estimate \eqref{eq:LevelSetProof:Cddd} from Lemma \ref{lma:TwoDimensionalConstants2} to obtain the bound $\frac{\nu(\widetilde{\cK})}{\nu(\cK)} \leq C_{ddd}$. In summary,
\begin{equation*}
\cK \in \GoodLambda \cup \BadLambdaND \quad \Rightarrow \quad \intdm{\cK \cap \{ H > \lambda \} }{H^{p'}}{\nu} \leq C_{ddd} \lambda^{p'} \nu(\cK)\,.
\end{equation*}
Using this inequality in conjunction with \eqref{eq:LevelSetProof:7} results in
\begin{equation*}
\intdm{\cB(x_0,\beta) \cap \{ H > \lambda \} }{H^{p'}}{\nu} \leq 10^{n+\veps p} \kappa^{p'} \lambda^{p'} \sum_j \nu(\cB_j) + C_{ddd} \lambda^{p'} \sum_{\cK \in \GoodLambda \cup \BadLambdaND} \nu(\cK)\,.
\end{equation*}
Then \eqref{eq:LevelSetProof:OffDiagConclusion} follows from the conclusions of Lemmas \ref{lma:EasyOffDiagSum} and \ref{lma:HardOffDiagSum}.
\end{proof}

\subsection{Conclusion of the Proof}\label{subsubsec:LevelSetConclusion}

We now come to the conclusion of the proof of Proposition \ref{prop:LevelSetEstimate}.
We start by combining the diagonal estimate \eqref{eq:LevelSetProof:28} with the off-diagonal estimate \eqref{eq:LevelSetProof:OffDiagConclusion}. We use the elementary estimate
\begin{equation*}
\intdm{\cB(x_0,\beta) \cap \{ H > \widetilde{\kappa} \kappa \lambda  \} }{H^{p'}}{\nu} \leq \lambda^{p'-\gamma} \intdm{\cB(x_0,\beta) \cap \{ H > \widetilde{\kappa} \kappa \lambda  \} }{H^{\gamma}}{\nu} + \intdm{\cB(x_0,\beta) \cap \{ H > \lambda  \} }{H^{p'}}{\nu}\,;
\end{equation*}
recall that $\widetilde{\kappa}, \kappa \in (0,1]$. Then \eqref{eq:LevelSetProof:28} and \eqref{eq:LevelSetProof:OffDiagConclusion}  give, after some elementary algebraic manipulations, the estimate
\begin{equation}\label{eq:Conclusion:Estimate1}
\begin{split}
\intdm{\cB(x_0,\beta) \cap \{ H > \widetilde{\kappa} \kappa \lambda \} }{H^{p'}}{\nu}
	&\leq \frac{C(\texttt{data})}{\veps^{2-2\gamma/p'} (\widetilde{\kappa} \kappa)^{p'-\gamma}  } (\widetilde{\kappa} \kappa \lambda)^{p'-\gamma} \intdm{\cB(x_0,\alpha) \cap \{ H > \widetilde{\kappa} \kappa \lambda \} }{H^{\gamma}}{\nu} \\
&\qquad + \frac{C_5 \lambda_1^{\vartheta_f}}{\wh{\kappa}^{p'} (\wh{\kappa} \kappa \lambda)^{\widetilde{\vartheta}_f-p'}} \intdm{\cB(x_0,\alpha) \cap \{ F > \wh{\kappa} \kappa \lambda \} }{F^{p_*}}{\nu}\,.
\end{split}
\end{equation}
Recall that $C_5$ is defined in \eqref{eq:LevelSetProof:27}. We can reformulate this estimate as
\begin{equation}\label{eq:Conclusion:Estimate2}
\begin{split}
\intdm{\cB(x_0,\beta) \cap \{ H > \lambda \} }{H^{p'}}{\nu}
	&\leq \frac{C}{\veps^{2-2\gamma/p'} (\widetilde{\kappa} \kappa)^{p'-\gamma}  } \lambda^{p'-\gamma} \intdm{\cB(x_0,\alpha) \cap \{ H > \lambda \} }{H^{\gamma}}{\nu} \\
&\qquad + \frac{C_6 \lambda_1^{\vartheta_f}}{\lambda^{\widetilde{\vartheta}_f-p'}} \intdm{\cB(x_0,\alpha) \cap \{ F > \wh{\kappa} \lambda / \widetilde{\kappa} \} }{F^{p_*}}{\nu}\,.
\end{split}
\end{equation}
The constants $C$ and $C_6$ satisfy the following dependencies:
\begin{equation*}
C \equiv C(\texttt{data})\,, \qquad C_6 \equiv C_6(\texttt{data},\veps)\,.
\end{equation*}
Since \eqref{eq:Conclusion:Estimate1} holds for all $\lambda \geq \lambda_2$ where $\lambda_2$ has been defined in \eqref{eq:DefnOfLambda2} we have that \eqref{eq:Conclusion:Estimate2} holds for all $\lambda \geq \widetilde{\kappa} \kappa \lambda_2$. Recall again that $\widetilde{\kappa}$, $\wh{\kappa}$, $\kappa \in (0,1]$ have been defined in \eqref{eq:LevelSetProof:19}, \eqref{eq:LevelSetProof:25} and \eqref{eq:KappaChoice} respectively.

In order to conclude with the level set estimate \eqref{eq:LevelSetEstimate} we need to estimate several constants. We need to obtain the specific dependence on $\veps$ of the constant appearing in front of the first integrand on the right-hand side. By using \eqref{eq:LevelSetProof:19} and  \eqref{eq:KappaChoice}, we see that we can choose $\widetilde{\kappa}$, $\kappa$ to satisfy
\begin{equation*}
\widetilde{\kappa} \kappa = \frac{\veps^{3/\gamma-2/p'}}{C}\,,
\end{equation*}
where $C \equiv C(\texttt{data})$. With this choice, we can estimate the constant appearing in front of the second integral in \eqref{eq:Conclusion:Estimate2} and therefore arrive at a choice of $\vartheta$ exactly as in \eqref{eq:DefnOfLevelSetConstants}.
We next set
\begin{equation*}
\kappa_f := \wh{\kappa} / \widetilde{\kappa}\,,
\end{equation*}
and making note of \eqref{eq:LevelSetProof:25} we can additionally choose $\wh{\kappa}$ small so that $\kappa_f \in (0,1)$.

Last, we need to find an upper bound on the numbers $\lambda_1$ and $\lambda_2$ defined in \eqref{eq:LevelSetProof:1} and \eqref{eq:DefnOfLambda2} respectively, so that the level set estimate \eqref{eq:LevelSetEstimate} can be verified for the range prescribed by \eqref{eq:DefnOfLambda0}. If $x \in B(x_0,\beta)$ and $\frac{\alpha-\beta}{40^n} \leq R \leq \frac{\varrho_0}{2}$, then $\cB(x,R) \subset \cB(x_0,2 \varrho_0)$. Recalling the doubling property \eqref{eq:MeasureDoublingProperty}, whenever $\widetilde{H}$ is a $\nu$-integrable function we can estimate
\begin{equation*}
\fint_{\cB(x,R)} \widetilde{H} \, \rmd \nu \leq \frac{\nu(\cB(x_0,2 \varrho_0))}{\nu(\cB(x,R))} \fint_{\cB(x_0,2 \varrho_0)} \widetilde{H} \, \rmd \nu \leq C \left( \frac{\varrho_0}{\alpha-\beta} \right)^{n+\veps p} \fint_{\cB(x_0,2 \varrho_0)} \widetilde{H} \, \rmd \nu
\end{equation*}
where $C = C(n)$. Applying this inequality to $H^{p'}$, $F^{p_*}$, $H^{\gamma}$, and $F^{p_*+\delta_f}$, as well as on different balls $2^k \cB(x,R) \subset 2^k \cB(x_0,2 \varrho_0)$, we get
\begin{equation}\label{eq:Conclusion:Estimate3}
\begin{split}
&\kappa^{-1} \big\{ \Psi_M(x,R) + \Upsilon_0(x,R) + Tail(x,R) \big\} \\
	&\qquad \leq \frac{C}{\veps^{1/\gamma}}  \left( \frac{\varrho_0}{\alpha-\beta} \right)^{n+\veps p} \big\{ \Psi_M(x_0,2 \varrho_0) + \Upsilon_0(x_0,2 \varrho_0) + Tail(x_0,2 \varrho_0) \big\} \\
	&\qquad \leq \frac{C}{\veps^{1/\gamma}}  \left( \frac{\varrho_0}{\alpha-\beta} \right)^{n+\veps p} \Theta(x_0,2 \varrho_0) \leq \frac{C}{\veps} \left( \frac{\varrho_0}{\alpha-\beta} \right)^{2n+p} \Theta(x_0,2 \varrho_0)\,,
\end{split}
\end{equation}
where $C \equiv C(\texttt{data})$. We also used \eqref{eq:KappaChoice} to remove dependence on $\kappa$, \eqref{eq:LevelSetProof:16} to remove dependence on $M$, that $\veps < 1$, and that $\frac{\varrho_0}{\alpha-\beta} \geq 2$. 
Recall also that $\Theta$ has been defined in \eqref{eq:DefnOfTheta}. Thus we have obtained the desired upper bound on $\lambda_1$. To estimate $\lambda_2$, note that for $\cK = K_1 \times K_2 \in \Delta_{k_0}$ with $k_0$ as in \eqref{eq:MinSideLength} we have $\cK \subset \cB(x_0,\alpha) \subset \cB(x_0,2 \varrho_0)$ and therefore
\begin{equation*}
\nu(\cK) \geq \frac{C}{\varrho_0^{n+\veps p}} \iintdm{K_1}{K_2}{}{y}{x} = \frac{C(\alpha-\beta)^{2n}}{\varrho_0^{n+\veps p}}\,, \qquad C \equiv C(n,p)\,.
\end{equation*}
Thus, for any cube $\cK \in \Delta_{k_0}$ we can estimate
\begin{equation}\label{eq:Conclusion:Estimate4}
\left( \fint_{\cK} H^{p'} \, \rmd \nu \right)^{1/p'} \leq \left( \frac{\nu(\cB(x_0,2 \varrho_0))}{\nu(\cK)}  \fint_{\cB(x_0,2 \varrho_0)} H^{p'} \, \rmd \nu \right)^{1/p'} \leq \frac{C}{\veps^{1/p'}} \left( \frac{\varrho_0}{\alpha-\beta} \right)^{2n/p'} \left( \fint_{\cB(x_0,2 \varrho_0)} H^{p'} \, \rmd \nu \right)^{1/p'}\,.
\end{equation}
Then using \eqref{eq:Conclusion:Estimate3} and \eqref{eq:Conclusion:Estimate4} (also using that $\veps < 1$ and $\frac{\varrho_0}{\alpha-\beta} \geq 2$) we get that
\begin{equation*}
\lambda_2 \leq \frac{C}{\veps} \left( \frac{\varrho_0}{\alpha-\beta} \right)^{2n+p} \Theta(x_0, 2 \varrho_0)\,,
\end{equation*}
where $C \equiv C(\texttt{data})$.
We therefore take $\lambda_0$ as in \eqref{eq:DefnOfLambda0} so that $\lambda_0 \geq \max \{ \lambda_1, \lambda_2 \}$. We finally arrive at \eqref{eq:LevelSetEstimate} with $\lambda$ prescribed as in \eqref{eq:DefnOfLambda0}, and the proof of Proposition \ref{prop:LevelSetEstimate} is complete.

\begin{remark}
Note that in the proof of Proposition \ref{prop:LevelSetEstimate} the off-diagonal analysis in Section \ref{subsec:OffDiag} does not make use of the assumption \eqref{eq:RevHolderFinal}. Thus, the conclusions of Lemma \ref{lma:OffDiagonalConclusion} hold for general fractional Sobolev functions $u$, and not just solutions to \eqref{eq:Intro:MainEqn}.
\end{remark}
%
%
%
%
%
%

\appendix

\end{document}